\documentclass[12pt]{amsart}
\usepackage{amsmath, amssymb, amsthm, esint, verbatim, hyperref, enumerate, enumitem}
\usepackage[letterpaper,margin=1.1in]{geometry}

\hypersetup{
  pdfpagelayout=SinglePage,
  pdfpagemode=UseOutlines,
  colorlinks,
  bookmarksopen,
  linkcolor=[rgb]{0,0,0.7},
  urlcolor=[rgb]{0,0,0.4},
  citecolor=[rgb]{0.4,0.1,0},
}

\newtheorem{theorem}[subsection]{Theorem}
\newtheorem{lemma}[subsection]{Lemma}
\newtheorem{corollary}[subsection]{Corollary}
\newtheorem{prop}[subsection]{Proposition}

\theoremstyle{definition}
\newtheorem{definition}[subsubsection]{Definition}
\newtheorem{remark}[subsection]{Remark}

\newcommand{\graph}{\mathrm{graph}}
\newcommand{\dom}{\mathrm{dom}}
\newcommand{\spt}{\mathrm{spt}}
\newcommand{\haus}{\mathcal{H}}

\newcommand{\proj}{\mathrm{proj}}
\newcommand{\cT}{\mathcal{T}}

\newcommand{\mcfM}{\mathcal{M}}
\newcommand{\dilD}{\mathcal{D}}
\newcommand{\refl}{\mathrm{refl}}

\newcommand{\sff}{\mathrm{II}}
\newcommand{\cBT}{\mathcal{BT}}
\newcommand{\eps}{\epsilon}
\newcommand{\spU}{\mathcal{U}}
\newcommand{\admis}{\mathcal{C}}
\newcommand{\R}{\mathbb{R}}

\title{The free-boundary Brakke flow}
\author{Nick Edelen}
\address{Department of Mathematics\\Massachusetts Institute of Technology\\Cambridge, MA, 02139-4307 }
\email{nedelen@mit.edu}

\begin{document}

\maketitle

\begin{abstract}
We develop the notion of Brakke flow with free-boundary in a barrier surface.  Unlike the classical free-boundary mean curvature flow, the free-boundary Brakke flow must ``pop'' upon tangential contact with the barrier.  We prove a compactness theorem for free-boundary Brakke flows, define a Gaussian monotonicity formula valid at all points, and use this to adapt the local regularity theorem of White \cite{white:local-reg} to the free-boundary setting.  We use Ilmanen's elliptic regularization procedure \cite{ilmanen:elliptic-reg} to prove existence of free-boundary Brakke flows.
\end{abstract}


\section{Introduction}

A surface $\Sigma$ has geometric free-boundary in a barrier hypersurface $S$ if $\partial\Sigma \subset S$, and $\Sigma$ meets $S$ orthogonally.  This is a physically and mathematically natural boundary condition to impose on geometric problems-with-boundary, and has garnered increasing interest over the past several years.

In the 90's Stahl \cite{stahl:regularity} proved long-time existence of the smooth, compact, free-boundary mean curvature flow of hypersurfaces, in the sense that curvature blow-up must occur at a finite-time singularity.  Some progress has been made analysing mean-convex singularities through smooth blow-ups: via a particular monotonicity formula Buckland \cite{buckland} proved type-I singularities are modeled on generalized cylinders (with free-boundary in a plane), and recently the author \cite{me:convexity} proved type-II singularities can be realized by translating solitons, via the Huisken-Sinestrari estimates.  Many others have considered smooth free-boundary curvature flows, including \cite{koeller}, \cite{stahl:singularity}, \cite{wheeler}, \cite{lambert:minkowski}, \cite{marquardt:thesis}.

A notion of free-boundary Brakke flow was originally written down by Mizuno-Tonegawa \cite{tonegawa:free-allen-cahn}, who proved existence of codimension-one free-boundary Brakke flows in convex barriers via the Allen-Cahn functional.  Recently Kagaya \cite{kagaya} extended their anaylsis to more general barriers.  The related notion of level-set flow with free-boundary has been extensively studied by several authors (see \cite{sato}, \cite{giga-sato}, \cite{katso-koss-reitich}, \cite{volkmann:thesis}).

In this paper we develop further the theory of Brakke flows with free-boundary.  Our first result is a compactness theorem for this class of flows (Theorem \ref{thm:brakke-compactness}).  The following simple example illustrates why one must modify Brakke's definition of flow to have compactness in the free-boundary setting: let $V_k$ be the regular $k$-gon inscribed in a circle.  Each $V_k$ is a stationary $1$-varifold, with free-boundary in the circle, and therefore by any reasonable definition is a Brakke flow with free-boundary.  However, taking $k \to \infty$, the $V_k \to S^1$ as varifolds, which is not anymore a Brakke flow!

However we are saved because ``in the directions tangential to $S^1$'' we are still a Brakke flow (in fact as free-boundary Brakke flows, the static $V_k$ converge to the empty flow; compare this to how sudden mass drop can occur in regular Brakke flows, even as limits of smooth flows).  As observed by Mizuno-Tonegawa, Brakke's original definition must be relaxed.   An interesting consequence is that this definition requires the flow to ``pop'' or ``break-up'' upon tangential contact with the barrier, which is in contrast to the smooth free-boundary mean curvature flow, for which the barrier is ``invisible.''

We then prove a monotonicity formula valid at all points (Theorem \ref{thm:monotonicity}), which gives an upper-semi-continuous notion of Gaussian density.  The formula given in \cite{buckland} is valid only along the boundary, and even for flat barriers only reduces to the standard Gaussian of the reflected flow when centered on the barrier.  We prove the existence of tangent flows (Theorem \ref{thm:tangent-flows}), which are always self-shrinkers either without boundary, or with free-boundary in a plane.  In the latter case one can reflect to obtain a self-shrinking Brakke flow without boundary.

Using our reflected Gaussian density we adapt the regularity arguments of White \cite{white:local-reg}, to prove an Brakke-type regularity theorem for free-boundary smooth flows, and limits of smooth flows (Theorem \ref{thm:brakke-reg}).  We further adapt Ilmanen's elliptic regularization to prove existence of the free-boundary Brakke flow, and smooth short-time existence from embedded, smooth initial data (Theorem \ref{thm:existence}).

Some of our results can be summarized in the following Theorem.
\begin{theorem}\label{thm:teaser}
Let $\Omega$ be a smooth domain in $\R^N$, with uniform $C^{3,\alpha}$ bounds on $\partial\Omega$, and take $\Sigma$ to be a smooth $n$-surface in $\Omega$ with geometric free-boundary in $\partial\Omega$.  Then there is a Brakke flow $(\mcfM(t))_{t \geq 0}$ in $\Omega$ with free-boundary in $\partial\Omega$, such that $\mcfM(0) = \Sigma$, and $\mcfM$ is smooth for short time.

Moreover, $\mcfM$ has the property that if any tangent flow  at a point $(x, t)$ is a multiplicity-1 plane (with possible free-boundary in a hyperplane), then near $(x, t)$ $\mcfM$ coincides with a smooth free-boundary flow.
\end{theorem}
The free-boundary Brakke flows constructed in Theorem \ref{thm:teaser} always stay to one side of the barrier, and have an associated current structure which rules out sudden vanishing (see Theorem \ref{thm:existence}).  If some interior point of $\mcfM$ hits the barrier, it must in infinitesimal time become a free-boundary component, i.e. ``pop''.  This is qualitatively different behavior from the classical immersed free-boundary flow; in the classical flow internal points of contact will continue flowing past the barrier.

I express my deepest gratitude to my advisors Simon Brendle and Brian White, and to my friend Otis Chodosh, for their  guidance and support.  I thank Masashi Mizuno for bringing several references to my attention.  This work borrows heavily from a series of lectures given by White at Stanford in Spring 2015, and Ilmanen's book on Elliptic Regularization \cite{ilmanen:elliptic-reg}.


\section{Preliminaries}

We live in $\R^N$.  Given a sequence of open sets $U_i$, $U$ in $\R^N$, we say $U_i \to U$ \emph{as open sets} if
\[
W \subset\subset U \iff W \subset \subset U_i \quad \forall i >> 1 ,
\]
for every precompact open $W$.  If the $U_i$, $U$ have smooth boundary, we say $U_i \to U$ in $C^{k,\alpha}$ if $U_i \to U$ as open sets, the $\partial U_i \to \partial U$ locally graphically in $C^{k,\alpha}$.

Let $S \subset \R^N$ be an embedded, oriented hypersurface, with orienting normal $\nu_S$.  We will always write $d(x)$ for the distance function to $S$.  An  $n$-surface $M$ \emph{meets $S$ orthogonally} if $\partial M \subset S$, and the outward conormal of $\partial M  \subset M$ coincides with $\nu_S$.  We will often say an $M$ meeting $S$ orthogonally has \emph{classical free-boundary} in $S$.

A family of immersions $F_t : M^n \times [0, T)$ is a \emph{classical mean curvature flow with free-boundary} in $S$ if
\begin{align*}
&\partial_t F_t(p) = H_t(p)  \text{, for all $p \in M$ and $t > 0$, and} \\
&F_t(M) \text{ meets $S$ orthogonally, for $t \geq 0$},
\end{align*}
here $H_t(p)$ being the mean curvature vector of $F_t(M)$ at $p$.

\subsection{Varifolds}

We will work extensively with integral varifolds.  Recall an $n$-varifold $V$ in $U \subset \R^N$ is a Radon measure on the Grassmanian $U \times Gr(n, N)$.  Here $Gr(n, N)$ is the space of unoriented $n$-planes in $\R^N$.  We denote the mass measure of $V$ by $\mu_V$, i.e. so $\mu_V = \pi_\sharp V$, with $\pi : U \times Gr(n, N) \to U$ being the projection.

We say $V$ is \emph{integral} if there is a collection of $C^1$ manifolds $N_i$, and subsets $S_i \subset N_i$, so that
\[
V(\phi(x, L)) = \sum_i \int_{S_i} \phi(x, T_x N_i) d\haus^n.
\]
The tangent plane $T_x N_i$ is well-defined $\haus^n$-a.e. in $S_i$.  Equivalently, $V$ is integral if
\[
\mu_V = \haus^n \llcorner \theta \llcorner M
\]
for some countably $n$-rectifiable set $M$, and some locally-$(\haus^n\llcorner M)$-integrable function $\theta$, taking non-negative integer values.  An integral varifold is uniquely determined by its mass measure.

Given a $C^1$ vector field in $\R^N$, and an $n$-plane $L$, write
\[
div_L(X) = tr_L(DX) = \sum_i <D_{e_i} X, e_i>,
\]
for any orthonormal basis $(e_i)$ of $L$.  For an integral varifold $V$, we write $div_V$ to mean $div_{T_x V}$, and $tr_V = tr_{T_x V}$, wherever the tangent plane exists.

The \emph{first variation} for mass of an integral varifold has the expression
\[
\delta V(X) = \int div_V(X) d\mu_V , \quad X \in C^1_c(U, \R^N) .
\]
We say $\delta V$ is locally finite if
\[
|\delta V(X)| \leq C_W |X|_{C^0} \quad \forall X \in C^1_c(W, \R^N) , W \subset\subset U
\]

If $\delta V$ is locally finite then we can define the total variation (Radon) measure
\[
||\delta V||(W) = \sup\{ \delta V(Y) : Y \in C^1_c(W, \R^N), |Y| \leq 1 \} \quad \forall W \subset\subset U ,
\]
and therefore differentiate to obtain
\[
\delta V(X)= -\int H \cdot X d\mu_V + \int X \cdot \nu_V d\sigma_V .
\]
Here $H_V = -\frac{d||\delta V||}{d \mu_V}$ is the \emph{generalized mean curvature vector}, $\sigma_V = ||\delta V||_{sing}$ is the \emph{generalized boundary measure}, and $\nu_V$ is the \emph{generalized outwards conormal}.

\subsection{Regularity scales and reflection}

Let $u$ be a function into $\R^{N-n}$, defined on some subset of $U \subset \R^n$.  We define the following (semi-)norms
\begin{align*}
|u|_{0, U} &= \sup_{x \in U \cap \dom(u)} |u(x)| \\
[u]_{\alpha, U} &= \sup_{x \neq y \in U \cap \dom(u)} \frac{|u(x) - u(y)|}{|x - y|^\alpha} \\
[u]_{k, \alpha, U} &= [D^k u]_{\alpha, U} , \\
|u|_{k,\alpha, U} &= [u]_{k,\alpha, U} + \sum_{i \leq k} |D^i u|_{0, U} .
\end{align*}
(This is not the scale-invariant $C^{k,\alpha}$-norm.)  Note we do not require $u$ to be defined on all of $U$.

Let $S$ be a smooth $n$-surface in $\R^N$, and $x \in S$.  We define the \emph{$C^k$ regularity scale} of $S$ at $x$, written as $r_k(S, x)$, to be the largest radius $r$ so that, after a suitable rotation, the translated and dilated surface
\[
(S - x)/r \cap (B^n_1 \times B^{N-n}_1)
\]
coincides with the graph of some $u : U \subset B^n_1 \to \R^{N-n}$, satisfying $|u|_{k, B^n_1} \leq 1$.  The \emph{$C^{k,\alpha}$ regularity scale} is defined in precisely the same manner, and is denoted by $r_{k,\alpha}(S, x)$.

Notice that $r_{2}$ bounds the inscribed radius from below, so a bound on $r_2$ is stronger than a bound on the second fundamental form.

Equivalently, the $C^k$ (or $C^{k,\alpha}$) regularity scale is the largest $r$ for which
\begin{equation}\label{eqn:reg-scale}
S \cap ((B^N_r(x) \cap L^n) \times (B^N_r(x) \cap L^\perp)) \subset \graph_L(u),
\end{equation}
where $L^n$ is some affine $n$-plane, and $u : U \subset L \to L^\perp$ is a function satisfying
\[
\sum_{i = 0}^k r^{i-1} |D^i u|_{0, U}  \quad \left( + r^{k+\alpha-1} [D^k u]_{\alpha, U} \,\, \text{ if $C^{k,\alpha}$} \right) \leq 1.
\]
If we take $L = T_x S$, then $u(0) = Du(0) = 0$, and we see that the regularity scale is always positive.  The best scale may not be achieved by $L = T_x S$ though.

We define the \emph{global regularity scale} $r_{k,\alpha}(S) = \inf_{x \in S} r_{k,\alpha}(S, x)$.

Throughout the rest of this paper we fix in notation $S^{N-1} \subset \R^N$ as our smooth, embedded, oriented \emph{barrier hypersurface}.  Write $\nu_S$ for the orienting unit normal, and $d(x)$ for the distance function to $S$.  Let $\zeta(x)$ be the nearest point projection of $x$ onto $S$.

We define the \emph{reflection across $S$} of a point $x$ to be
\[
\tilde x = 2\zeta(x) - x.
\]
Given a vector $v \in T_x \R^N$, we let $\tilde v \in T_{\tilde x} \R^N$ be the \emph{linear} reflection across $T_{\zeta(x)}S$.  In other words, if $\refl$ denotes the usual affine reflection across $T_{\zeta(x)}S$, then
\[
\tilde v = D(\refl)(v) \in T_{\tilde x} \R^N.
\]
We shall only consider $\zeta$ and $\tilde x$ as defined in sufficiently small neighborhood of $S$, so both are smooth functions of $x$.

Pick a point $y \in S$, and for simplicity assume $y = 0$, and $T_y S = \R^{N-1} \times \{0\}$.  Let $r_{3}(S, y) = \rho$, and take $u : B^{N-1}_\rho(0) \to \R$ to be the graph in \eqref{eqn:reg-scale} realizing $r_{3}$.  We define the \emph{inverse projection map} at $y$ to be
\begin{align}
\Phi(x', s) &= (x', u(x')) + s \nu_S(x', u(x')) \label{eqn:straighten}\\
&= (x', u(x')) + \frac{s(-Du(x'), 1)}{\sqrt{1+|Du(x')|^2}} \nonumber.
\end{align}

One can check directly (since $\Phi(0) = 0$, $D\Phi(0) = Id$) that, for any $r \leq \rho$, on $B^{N-1}_r(0) \times [-r, r]$ we have the estimates
\begin{equation} \label{eqn:reflect-bounds}
|\Phi - Id| \leq c(n) \frac{r^2}{\rho}, \quad |D\Phi - Id| \leq c(n) \frac{r}{\rho}, \quad |D^2 \Phi|_0 \leq \frac{c(n)}{\rho} .
\end{equation}
If we further have $r_{k,\alpha}(S, y) = \rho$, then
\[
\rho^{k+\alpha-1} [D^k \Phi]_{\alpha, B_\rho^{N-1} \times [-\rho, \rho]} \leq c(n, k, \alpha).
\]

We make the following definition out of convenience.  Let $c_0(n) = 10c(n)$, where $c(n)$ as in \eqref{eqn:reflect-bounds}.  For $x \in S$, define the \emph{reflection regularity scale} $r_S(x)$ to be the largest radius $\leq r_3(S, x)$ so that:
\begin{enumerate}
\item[A)] on $B^N_{r_S}(x)$ the inverse $\Phi^{-1}$ centered at $x$ exists, and satisfies
\[
|\Phi^{-1} - Id|(z) \leq c_0 \frac{|z - x|^2}{r_S}, \quad |D\Phi^{-1} - Id|(z) \leq c_0 \frac{|z - x|}{r_S}, \quad |D^2\Phi^{-1}|(z) \leq \frac{c_0}{r_S}; 
\]

\item[B)] on $B^{N-1}_{r_S}(x) \times [-r_S, r_S]$ (having identified $T_x S$ with $\R^{N-1}$), $\Phi$ satisfies
\[
|\Phi - Id|(z) \leq c_0 \frac{|z - x|^2}{r_S}, \quad |D \Phi - Id|(z) \leq c_0\frac{|z - x|}{r_S}, \quad |D^2\Phi|(z) \leq \frac{c_0}{r_S} .
\]
\end{enumerate}

The \emph{global reflection regularity scale} of $S$ is defined to be $r_S = \inf_{x \in S} r_S(x)$.  By construction $r_S$ scales with $S$, and in fact by \eqref{eqn:reflect-bounds} we can take $r_S(x) = \epsilon(n) r_3(S, x)$.

We require with some further estimates on $\Phi$.  We remark that $d(x) \leq |x| + |\tilde x|$, where defined.
\begin{lemma}\label{lem:refl-bounds}
Let $x \in B_{r_S}(S)$, and $v \in T_x \R^N$ be a unit vector.  Write $\refl$ for the affine reflection about $T_{\zeta(x)}S$.  We have
\[
|\tilde y - \refl(y)| \leq c_1(n)\frac{|y - \zeta(x)|^2}{r_S} \text{ on } B_{r_S}(\zeta(x)), \quad |D_v \tilde x - \tilde v| \leq c(n) \frac{d(x)}{r_S}, \quad |D_v D_v \tilde x| \leq \frac{c(n)}{r_S}.
\]
For any $n$-plane $L^n$, at $x$ we have
\[
\left| \mathrm{tr}_L D^2|\tilde x|^2 - 2n \right| \leq c(n) \left(\frac{d(x)}{r_S} + \frac{|\tilde x|}{r_S} \right).
\]
Recall that $d(x)$ is the distance function to $S$.
\end{lemma}

\begin{proof}
Let $\Phi$ be the inverse projection map at $\zeta(x)$.  For any $y \in B_{r_S}(\zeta(x))$, we have the relation
\begin{align*}
\tilde y &= \Phi(\refl(\Phi^{-1}(y))) \\
&= \refl(y) + (\Phi - Id)(\refl(\Phi^{-1}(y))) + \refl((\Phi^{-1} - Id)(y)) .
\end{align*}

Therefore, recalling that $\Phi(\zeta(x)) = \zeta(x)$, we obtain 
\begin{align*}
|\tilde y - \refl(y)| 
&\leq c(n) ( \sup_{B_{r_S}(\zeta(x))} |D\Phi^{-1}| ) |y - \zeta(x)|^2/r_S + c(n) |y - \zeta(x)|^2/r_S \\
&\leq c(n) |y - \zeta(x)|^2/r_S,
\end{align*}
and
\begin{align*}
|D_v \tilde x - \tilde v|
&= |D_v \tilde x - D(\refl)(v)| \\
&\leq |(D\Phi - Id)|_{\refl(\Phi^{-1}(x))}| |D\Phi^{-1}|_x| + |(D\Phi^{-1} - Id)|_x| \\
&\leq c(n) |x - \zeta(x)|/r_S,
\end{align*}
and
\begin{align*}
|D_v D_v \tilde x| \leq |D^2 \Phi| |D\Phi|^2 + |D\Phi - Id| |D^2\Phi^{-1}| + |D^2\Phi^{-1}| \leq c(n)/r_S .
\end{align*}
The last relation follows directly from the others.  
\end{proof}

\subsection{Spacetime and flows}\label{section:spacetime-flows}

We define \emph{spacetime} to be space
\[
\R^{N,1} = \{ (x, t) : x \in \R^N, t \in \R \} \equiv \R^{N+1},
\]
endowed with the parabolic norm
\[
|X| \equiv |(x, t)| = \max\{|x|, |t|^{1/2} \}.
\]
We typically use capitals to denote points in spacetime.  We define the spacetime ball  centered at $X$ to be
\[
B^{N,1}_R(X) = \{ Y \in \R^{N,1} : |X - Y| < R \}.
\]
In $\R^{N,1}$ time naturally scales like space squared.  We write the parabolic scaling operator as
\[
\dilD_\lambda(x, t) = (\lambda x, \lambda^2 t) .
\]

Given a function $u : U \subset B^{n, 1}(0) \to \R^{N-n}$, the \emph{spacetime graph} of $u$ is the set
\[
\graph(u) = \{ (u(x, t), x, t) : x \in U, t \in \R \} .
\]
We define the parabolic (semi-)norms
\begin{align*}
[u]_{\alpha, U} &= \sup_{X \neq Y \in U \cap \dom(u)} \frac{|u(X) - u(Y)|}{|X - Y|^\alpha}, \\
[u]_{k,\alpha, U} &= \sum_{i + 2j = k} [D^i \partial_t^j u]_{\alpha, U} ,\\
|u|_{k,\alpha, U} &= [u]_{k,\alpha,U} + \sum_{i + 2j \leq k} |D^i \partial_t^j u|_{0, U} .
\end{align*}
Here $|X|$ is the parabolic norm.

Let $\mcfM$ be a $C^\infty$ submanifold (with possible boundary) of $\R^{N+1}$ in the ordinary Euclidean sense, having (Euclidean) dimension $n+1$.  We define the \emph{parabolic $C^{k,\alpha}$-regularity scale} of $\mcfM$ at $X$, written $r_{k,\alpha}(\mcfM, X)$, to be the largest radius $r$ so that, after a suitable rotation in space, the dilated and translated submanifold
\[
\dilD_{1/r}(\mcfM - X) \cap (B^{N-n}_1 \times B^{n,1}_1) = \graph(u)
\]
for some $u : U \subset B^{n,1}_1 \to \R^{N-n}$ having $|u|_{k,\alpha, B^{n,1}_1} \leq 1$.  If no such $r$ exists we set $r_{k,\alpha} = 0$.

However if $X$ is not a critical point for the time function $(x, t) \mapsto t$ restricted to $\mcfM$, then due to the parabolic scaling $r_{k,\alpha}(\mcfM, X)$ will be positive.  In fact $r_{k,\alpha}$ is bounded away from $0$ on compact subsets of $\mcfM$, since if $r_{k,\alpha}(\mcfM, X) > 0$ then one can choose an $0 < r \leq r_{k,\alpha}(\mcfM, X)$ so that $r_{k,\alpha} \geq r$ in some neighborhood of $X$.

We say $\mcfM$ is a \emph{smooth flow with classical free-boundary} in $S$, if for some $T \in (-\infty, \infty]$, 
\begin{enumerate}
\item[A)] $r_{2,\alpha}(\mcfM, X) > 0$ at every $X \in \mcfM$;

\item[B)] $\partial \mcfM \subset (S \times (-\infty, T]) \cup \{ t = T \}$;

\item[C)] each slice
\[
\mcfM(t) = \{ x \in \R^N : (x, t) \in \mcfM \}
\]
meets $S$ orthogonally.
\end{enumerate}
We often refer to $\mcfM$ as the \emph{spacetime track} of the flow.  Given an open set $\spU \subset \R^{N,1}$, we say $\mcfM$ is \emph{proper} in $\spU$ if
\[
\overline{\mcfM} \cap \spU = \mcfM \cap \spU ,
\]
here $\overline{\mcfM}$ denoting the set-theoretic closure of $\mcfM$.

We say $\mcfM$ is a smooth \emph{mean curvature flow}, with classical free-boundary in $S$, if it satisfies conditions A)-C), and additionally
\begin{enumerate}
\item[D)] the time-slices $(M_t)_t \equiv (\mcfM(t))_t$ move by free-boundary mean curvature flow.
\end{enumerate}

If a family of immersions $F_t : M^n \times [0, T) \to \R^N$ defines a classical mean curvature flow, with free-boundary in $S$, then the spacetime track 
\begin{equation}
\mcfM = \{ (x, t) : x \in F_t(M), t \in (0, T) \}
\end{equation}
will be a smooth mean curvature flow in the above definition.

\section{Free-boundary varifolds}

Fix $S$ a smooth, embedded, oriented barrier hypersurface in $\R^N$, with $r_S > 0$.  Write $\nu_S$ for the orienting normal.  We will always write $V$ for an integral $n$-varifold, with $n < N$.

\begin{definition}
Given an open $U \subset \R^N$, write $\cT(S, U)$ for the space of vector fields $X \in C^0_c(U, \R^N)$ which lie tangent to $S$, i.e. $X|_S \subset TS$.  We abbreviate $\cT(S) \equiv \cT(S, \R^N)$.
\end{definition}
\begin{definition}
Given a Borel-measurable vector field $X$, define the (discontinuous) vector fields
\[
S^\perp(X) = 1_{S} (X \cdot \nu_S) \nu_S, \quad S^T(X) = X - S^\perp(X).
\]
Here $\nu_S$ is the orienting normal.  Observe that $S^T(X)$ lies tangent to $S$, and each $S^\perp(X)$, $S^T(X)$ is Borel-measurable.
\end{definition}

\begin{definition}
Let $V$ be an integral $n$-varifold in $U \subset \R^N$.  We say $V$ has \emph{free-boundary in $S \subset U$} if
\begin{equation}\label{eqn:fb-varifold}
\delta V(X) = - \int S^T(H) \cdot X d\mu_V \quad \forall X \in \cT(S, U) \cap C^1,
\end{equation}
for some $S^T(H) \in L^1_{loc}(U, \R^N; \mu_V)$.
\end{definition}
Of course when the mean curvature vector exists then $H \cdot X = S^T(H) \cdot X$ for every $X \in \cT(S, U)$.  We write $S^T(H)$ in \eqref{eqn:fb-varifold} to emphasize that $S^T(H)$ is the natural quantity for free-boundary varifolds.

As with classical free-boundary manifolds, if $S$ is a hyperplane then $V$ can be reflected to a varifold without boundary.
\begin{prop}\label{prop:reflect-varifold}
Let $V$ have free-boundary in a hyperplane $P \subset \R^N$, and write $A : \R^N \to \R^N$ for the reflection about $P$.  Then the varifold $\tilde V = V + A_\sharp V$ satisfies
\begin{equation}\label{eqn:reflect-varifold-var}
\delta \tilde V(Y) = -\int S^T(H_V) \cdot Y d\mu_V -  \int (A \circ S^T(H_V) \circ A) \cdot Y d\mu_{A_\sharp V} .
\end{equation}
So $||\delta \tilde V|| \leq ||\delta V|| + ||\delta (A_\sharp V)||$, and $\tilde V$ has no generalized boundary.

Moreover, for any precompact $W$ have
\begin{equation}\label{eqn:reflect-varifold-L2}
\int_W |H_{\tilde V}|^2 d\tilde\mu \leq \int_W |S^T(H_V)|^2 d\mu_V + |S^T(H_V)\circ A|^2 d\mu_{A_\sharp V} .
\end{equation}
\end{prop}

\begin{proof}
By direct calculation we have
\[
(div_{A(L)}(Y))(A(x)) = (div_{L}(A\circ Y \circ A))(x).
\]
Since $Y + A \circ Y \circ A \in \cT(S)$, we therefore obtain
\begin{align*}
\delta\tilde V(Y) &= \delta V(Y + A \circ Y \circ A) \\
&= -\int S^T(H_V) \cdot (Y + A \circ Y \circ A) d\mu_V \\
&= -\int S^T(H_V) \cdot Y d\mu_V -  \int (A \circ S^T(H_V) \circ A) \cdot Y d\mu_{A_\sharp V} .
\end{align*}

To deduce inequality \eqref{eqn:reflect-varifold-L2}, we use that for any $X \in C^1_c(W)$, 
\begin{align*}
&\int H_{\tilde V} \cdot X d\tilde \mu  \\
& \leq \left( \int_W |S^T(H_V)|^2 d\mu_V \right)^{1/2} \left( \int |X|^2 d\mu_V \right)^{1/2} \\
&\quad + \left( \int_W |S^T(H_V)\circ A|^2 d\mu_{A_\sharp V} \right)^{1/2} \left( \int |X|^2 d\mu_{A_\sharp V} \right)^{1/2} \\
&\leq \left( \int_W |S^T(H_V)|^2 d\mu_V + |S^T(H_V)\circ A|^2 d\mu_{A_\sharp V} \right)^{1/2} \left( \int |X|^2 d\tilde \mu \right)^{1/2} . \qedhere
\end{align*}
\end{proof}

We shall prove that any free-boundary varifold has locally bounded total variation.  Here is the intuition.  Suppose $V$ were smooth up to the barrier, and we have control over $S^T(H)$.  Almost-everywhere in $\spt(V) \cap S$ we have $T_xV \subset T_x S$, and $\proj_{N_x S} \circ \sff_V = {\sff_S}|_{T_x V}$.  So we have control over $S^\perp(H)$ also.  Then using a trace formula we obtain control over $||\partial V||$.

For general integral varifolds we accomplish this using a monotonicity formula due to Allard.
\begin{prop}[compare Allard \cite{allard:first-variation}]\label{prop:fb-finite-var}
If $V$ is a free-boundary varifold in $U$, then $||\delta V||$ is locally finite in $U$, and for every $W \subset\subset W' \subset\subset U$ we have
\begin{equation}\label{eqn:boundary-var}
||\delta V||(W) \leq c(n) \int_{W'} |S^TH| d\mu_V + c(n) \left( \frac{1}{r_S} + \frac{1}{d(W, \partial W')} \right) \mu_V(W') .
\end{equation}
(actually we can replace $r_S$ with $r_2(S)$).

Therefore a locally-finite integral varifold $V$ has free-boundary in $S$ if and only if $V$ has locally bounded first-variation, and the generalized boundary measure $\sigma_V$ is supported in $S$, and $\nu_V = \nu_S$ at $\sigma_V$-a.e. $x$.
\end{prop}

\begin{proof}
From the boundary monotonicity formula (see \ref{prop:boundary-mono}) we have that, for any $h \in C^1_c(W, \R)$, 
\begin{align}
&\Gamma(h) := \lim_{\tau \to 0} \tau^{-1} \int_{B_\tau(S)} h |D^T d|^2 d\mu_V \nonumber \\
&= r_S^{-1} \int_{B_{r_S}(S)} h |D^T d|^2 \nonumber \\
&\quad - \int_{B_{r_S}(S) \setminus S}  (1 - d/r_S) (D^T h \cdot D^T d + h (tr_V D^2 d) + h (S^TH \cdot D d) ) d\mu_V . \label{eqn:boundary-mono}
\end{align}

Letting $g$ be a function satisfying:
\[
0 \leq g \leq 1, \quad g \equiv 1 \text{ on } W, \quad \spt(g) \subset W' , \quad |Dg| \leq 10 / d(W, \partial W').
\]
By \eqref{eqn:boundary-mono} the limit $\Gamma(h)$ always exists, and by our construction of $g$ we have $|h| \leq |h|_{0} g$.  Since $|D^2 d| \leq c(n)/r_S$ on $B_{r_S}(S)$, we deduce
\begin{align}
|\Gamma(h)| &\leq |h|_0 \Gamma(g) \nonumber \\
&\leq |h|_0 \left( \int_{W'} |S^TH| d\mu_V + c(n) (r_S^{-1} + d(W, \partial W')^{-1})\mu_V(W') \right) .\label{eqn:boundary-mono-limit}
\end{align}

Given $X \in C^1_0(W, \R^N)$, define in $B_{r_S}(S)$ the vector fields
\[
X^{S\perp} = (X \cdot D d) D d, \quad X^{ST} = X - X^{S\perp} ,
\]
and let $\eta : \R_+ \to \R_+$ be a function satisfying
\[
\eta = 0 \text{ on } [0, \rho/2], \quad \eta = 1 \text{ on } [\rho, \infty), \quad 0 \leq \eta' \leq 10/\rho ,
\]
for $\rho < r_S$.

We calculate
\begin{align*}
\int div_V(X) d\mu_V 
&= \int div_V(\eta(d) X + (1-\eta(d)) X^{ST} + (1-\eta(d))X^{S\perp}) d\mu_V \\
&= -\int S^T(H) \cdot (\eta(d) X + (1-\eta(d))X^{ST}) d\mu_V \\
&\quad + \int (-\eta') (X \cdot D d) |D^T d|^2 d\mu_V + \int (1-\eta) div_V(X^{S\perp}) d\mu_V .
\end{align*}

We bound the last two terms.  Taking $\rho \to 0$, we have
\begin{align*}
\lim_{\rho \to 0} \text{(penultimate term)}
&= \lim_{\rho \to 0} \rho^{-1} \int_{B_{\rho}(S)} (-\rho \eta' (X \cdot D d)) |D^T d|^2 d\mu_V \\
&= \pm 10 C |X|_0,
\end{align*}
using \eqref{eqn:boundary-mono-limit}, where we simply write $C$ for the expression in brackets.

By considering a countable $C^1$ cover of the underlying rectifiable set, we have $T_x V \subset T_x S$ for $\mu_V$-a.e. $x \in S$.  Therefore by the dominated convergence theorem we have
\begin{align*}
\lim_{\rho \to 0} \text{(last term)}
&= \int_S div_V((X \cdot \nu_S) \nu_S) d\mu_V \\
&= \int_S (X \cdot \nu_S) div_V(\nu_S) d\mu_V \\
&= \pm c(n) r_S^{-1} |X|_0 \mu_V(W) .
\end{align*}
This proves equation \eqref{eqn:boundary-var}.

We prove the equivalence assertion.  The ``if'' direction is clear.   Conversely, the above show $||\delta V||$ is Radon.  The free-boundary condition \eqref{eqn:fb-varifold} trivially shows $||\delta V|| << \mu_V$ on $\R^N \setminus S$, and since $S$ is closed we have $\spt(\sigma_V) \subset S$.

Given any $Y \in C^1_c(B_{r_S}(S), \R^N)$, we have using the notation above
\begin{align*}
\delta V(Y)
&= \delta V(Y^{ST} + Y^{S\perp}) \\
&= -\int H \cdot Y^{ST} d\mu_V + \int \nu_V \cdot S^\perp(Y) d\sigma_V .
\end{align*}
Therefore
\[
\int S^T(\nu_V) \cdot Y d\sigma_V = 0  \quad \forall Y \in C^1_c(\R^N, \R^N) .
\]

Given any $W \subset\subset \R^N$ we can approximate $\nu_V$ in $L^1(W, \R^N; \sigma_V)$ by $Y \in C^1_c(W, \R^N)$, to deduce $S^T(\nu_V) = 0$ $\sigma_V$-a.e.
\end{proof}

We have the following immediate Corollary.
\begin{corollary}\label{cor:uniform-fb-bounds}
Let $V_i$ be a sequence of free-boundary varifolds with boundary in $S_i \subset U_i$.  Suppose $U_i \to U$ and $\inf r_{S_i} > 0$.  If
\[
\sup_i \int_W 1 + |S^T(H_i)| d\mu_{V_i} \leq C_1(W) \quad \forall W \subset\subset U,
\]
then
\[
\sup_i ||\delta V_i||(W) \leq D_1(W) \quad \forall W \subset\subset U .
\]
Here $D_1$ depends only on $n, \inf r_{S_i}, C_1, W, U$.
\end{corollary}

\vspace{5mm}

We wish to prove a compactness Theorem for free-boundary varifolds.  We require some initial approximation results.

\begin{prop}\label{prop:fb-density-theorems}
Let $U \subset \R^N$ be an open set.  We have the following density theorems.
\begin{enumerate}
\item[A)] The space $\{ \phi \in C^\infty_c(U, \R) : D\phi \in \cT(S, U) \}$ is dense in $C^0_c(U, \R)$.

\item[B)] If $\mu$ is finite and rectifiable on $U$, and $1 \leq p < \infty$, then the $L^p(\mu)$ closure of $\cT(S, U) \cap C^1$ is $S^T(L^p(U, \R^N; \mu))$.
\end{enumerate}
\end{prop}

\begin{proof}
Part A) is clear.  We prove part B).  Fix an $X \in L^p(\mu)$, and since $\mu$ is rectifiable (see Lemma 7.2 in \cite{ilmanen:elliptic-reg}) we can choose $Y_1 \in C^1_c(U, \R^N)$ so $||Y_1 - X||_{L^p(\mu)} < \epsilon$.  Now choose $Y_2 \in C^1_c(S, TS)$ so that $||Y_2 - S^T(X)||_{L^p(\mu \llcorner S)} < \epsilon$, which we can do since $\mu \llcorner S$ is rectifiable also.  Let
\[
\Lambda = \max\{ |Y_1|_0, |Y_2|_0, 1\} , \quad B_R \supset \spt Y_1 \cup \spt Y_2.
\]

We can pick $\rho$ so that
\[
\mu (B_R \cap (B_\rho(S) \setminus S)) < \epsilon / \Lambda , \quad ||X|_{B_R \cap (B_\rho(S) \setminus S)} ||_{L^p(\mu)} < \epsilon.
\]
Let $\tilde Y_1$ be a $C^1_c$ extension of ${Y_1}|_{B_R \setminus B_\rho(S)}$, and $\tilde Y_2$ a $C^1_c$ extension of $Y_2$, and we can ensure
\[
\max\{|\tilde Y_1|_0, |\tilde Y_2|_0\} \leq \Lambda, \quad Y_i = \tilde Y_i \text{ away from } B_R \cap (B_\rho(S) \setminus S) .
\]

Then $\tilde Y = \tilde Y_1 + \tilde Y_2 \in \cT(S)$, and 
\begin{align*}
||S^T(X) - \tilde Y||_{L^p(\mu)}
&\leq 2\epsilon + ||X|_{B_R \cap (B_\rho(S) \setminus S)}|| + ||Y_1 - \tilde Y_1|| + ||Y_2 - \tilde Y_2|| \\
&\leq 3\epsilon + 4\Lambda \cdot \mu(B_R \cap (B_\rho(S) \setminus S)) \\
&\leq 10 \epsilon .
\end{align*}

This shows $\cT(S)$ is dense in $S^T(L^p(\mu))$.  Conversely, given $X \in L^p(\mu)$, then $S^T(X) = X$ iff
\[
\int_{S \cap W} X \cdot \nu_S d\mu = 0 \quad \forall W \subset\subset \R^N .
\]
The ``only if'' part is trivial, and the ``if'' by differentiation.  Since the above relation is preserved under $L^p(\mu)$ limits, $S^T(L^p(\mu))$ is closed.
\end{proof}

\begin{theorem}\label{thm:varifold-compactness}
Let $V_i$ have free-boundary in $S_i \subset U_i$.  Suppose $\inf r_{S_i} > 0$, $U_i \to U$, and $S_i \to S$ in $C^3_{loc}$.  Suppose
\[
\sup_i \int_{W} 1 + |S_i^T(H_i)|^2 d\mu_{V_i} \leq C_1(W) \quad \forall W \subset\subset U.
\]

Then there is an integral $n$-varifold $V$ with free-boundary in $S \subset U$, such that after passing to a subsequence, we have $V_i \to V$ as varifolds, $S_i^T(H_{V_i}) d\mu_{V_i} \to S^T(H_V) d\mu_V$ as Radon measures on $\cT(S, U)$, and
\[
\int_W |S^T(H_V)|^2 \leq \liminf_i \int_W |S^T(H_{V_{i'}})|^2 \quad \forall W \subset \subset U.
\]

In particular, if $\phi_i, \phi$ are $C^1_c$ functions with uniformly bounded supports, with $D\phi \in \cT(S, U)$ and $D \phi_i \to D\phi$ in $C^0$, then we have
\begin{gather}\label{eqn:H-lsc}
\int -|S^T(H_V)|^2 \phi + S^T(H_V) \cdot D\phi d\mu_V \geq \limsup_i \int -|S_i^T(H_{V_i})|^2 \phi_i + S_i^T(H_{V_i}) \cdot D\phi_i d\mu_{V_i}.
\end{gather}
\end{theorem}

\begin{proof}
By Holder's inequality we have uniform $L^1_{loc}(\mu_{V_i})$ bounds on $S_i^T(H_{V_i})$, and hence by Corollary \ref{cor:uniform-fb-bounds} (and that $r_{S_i}$ is uniformly bounded below) we have uniform local bounds on $||\delta V_i||$.  Allard's compactness theorem implies subsequential convergence to some integral $n$-varifold $V$ in $U$, with locally finite variation.

Fix a precompact $W \subset\subset U$.  Choose any $X \in \cT(S, W) \cap C^1$.  We can find a sequence $X_i \in \cT(S_i, W) \cap C^1$ so that $X_i \to X$ in $C^1$.  We have that $\delta V_i(X_i) \to \delta V(X)$ and $\mu_{V_i}(|X_i|^2) \to \mu_V(|X|^2)$.

For each $i$ we have
\[
|\delta V_i(X_i)| \leq \left( \int_W |S_i^T(H_i)|^2 d\mu_{V_i} \right)^{1/2} \left( \int |X_i|^2 d\mu_{V_i} \right)^{1/2} .
\]
Take the limit on both sides, to deduce
\[
|\delta V(X)| \leq \left( \liminf_i \int_W |S_i^T(H_i)|^2 d\mu_{V_i} \right)^{1/2} \left( \int |X|^2 d\mu_V \right)^{1/2} .
\]

Therefore, using Proposition \ref{prop:fb-density-theorems}, we deduce $\delta V$ is an $L^2(\mu_V)$ operator on $S^T(L^2(W, \R^N; \mu_V))$.  So by Proposition \ref{prop:fb-finite-var} $V$ has free-boundary in $S$, and
\[
\int_W |S^T(H_V)|^2 d\mu_V = ||\delta V|_{\cT(S, W)}||(W) \leq \liminf_i \int_W |S^T(H_i)|^2 d\mu_i .
\]

Given any $X \in \cT(S, U)$, we can approximate $X$ by elements of $\cT(S, U) \cap C^1$, and $\cT(S_i, U_i) \cap C^1$ as above, to deduce
\[
\int S_i^T(H_{V_i})\cdot X d\mu_{V_i} \to \int S^T(H_V) \cdot X d\mu_V.
\]

Let us prove \eqref{eqn:H-lsc}.  Using the above, standard layer-cake formulas, and Fatou, we have
\begin{align*}
\int |S^T(H_V)|^2 \phi d\mu_V &= \int_0^\infty \int_{\{\phi > s\}} |S^T(H_V)|^2 d\mu_V ds \\
&\leq \int_0^\infty \liminf_{i} \int_{\{\phi > s\}} |S^T(H_{i})|^2 d\mu_{i} ds \\
&\leq \liminf_{i} \int |S^T(H_{i})|^2 \phi d\mu_{i}  = \liminf_{i} \int |S^T(H_{i})|^2 \phi_{i'} d\mu_{i} . \qedhere
\end{align*}
\end{proof}

\section{Free-boundary Brakke flows}

\subsection{Definition, basic properties}

For the duration of this paper we adopt the notation that $\tau \equiv -t$.

\begin{definition}
Given an open $U \subset \R^N$, and an interval $I \subset \R$, let $\cBT(S, U, I)$ be the set of non-negative functions
\[
\{ \phi \in C^1(U \times I, \R_+) : D\phi(\cdot, t) \in \cT(S, U) \quad \forall t \in I \} .
\]
This is our set of admissible test functions.  When there is no ambiguity we may omit the $U$ or $I$.
\end{definition}

\begin{definition}\label{def:fb-brakke}
Let $I \subset \R$ be some interval.  We say a collection $(\mu(t))_{t \in I}$ of Radon measures is an $n$-dimensional \emph{Brakke flow in $U$ with free-boundary in $S$} if the following holds:
\begin{enumerate}
\item[A)] for a.e. $t \in I$, $\mu(t) = \mu_{V(t)}$ for some integral $n$-varifold $V(t)$ with free-boundary in $S \subset U$, having
\[
S^T(H_{V(t)}) \in L^2_{loc}(U, \R^n; \mu_{V(t)});
\]

\item[B)] for any finite interval $[a, b] \subset I$, and every $\phi \in \cBT(S, U, [a, b])$, we have that the mapping
\[
t \mapsto \int -|S^T(H_{V(t)})|^2 \phi + S^T(H_{V(t)}) \cdot D\phi + \partial_t \phi d\mu_{V(t)}
\]
(defined for a.e. $t$) is measurable on $[a, b]$, and
\begin{equation}\label{eqn:fb-evolution}
\int \phi(\cdot, b) d\mu(b) - \int \phi(\cdot, a) d\mu(a) \leq \int_a^b \int -|S^T(H)|^2 \phi + S^T(H) \cdot D\phi + \partial_t \phi d\mu(t) dt .
\end{equation}
\end{enumerate}
Here $I$ is the \emph{time interval of definition}.

Given a domain $\Omega$, a free-boundary Brakke flow is \emph{supported in $\Omega$} if it additionally satisfies:
\begin{enumerate}
\item[C)] for a.e. $t$, $\spt\mu(t) \subset \overline{\Omega}$.
\end{enumerate}

If $S = \partial\Omega$, and $(\mu(t))_t$ is supported in $\Omega$, for short we will sometimes say $(\mu(t))_t$ is a \emph{free-boundary Brakke flow in $\Omega \subset U$}.
\end{definition}

Since the free-boundary condition only sees vector fields parallel to $S$, the natural curvature becomes $S^T(H)$ instead of $H$.  As demonstrated in the introductory example, in general $S^\perp(H)$ is poorly behaved in limits.  On the other hand, Proposition \ref{prop:brakke-reflect} illustrates why we can expect definition \eqref{eqn:fb-evolution} to still admit good regularity.

As far as the dynamics are concerned we are effectively ``modding out'' by $S$.  For example, the measures $t \mapsto \haus^{N-1} \llcorner S$ form a vacuous free-boundary Brakke flow, and adding $\haus^{N-1}\llcorner S$ to any free-boundary Brakke flow gives a free-boundary Brakke flow with identical dynamics.

\begin{remark}
Any classical mean curvature flow with free-boundary $(M^n_t)_t$ is a free-boundary Brakke flow, by taking $\mu(t) = \haus^n\llcorner M_t$.
\end{remark}

\begin{remark}
Any (free-boundary) Brakke flow on $t \in [a, b)$ or $t \in [a, b]$ can trivially be extended to times $[a, \infty)$ by setting $\mu(t) = 0$ for all $t \geq b$ (resp. $t > b$).
\end{remark}

\begin{remark}[A remark on scaling]\label{rem:scaling}
As with smooth mean curvature flows, any (free-boundary) Brakke flow can be translated or parabolically dilated in spacetime to obtain a new Brakke flow.  Precisely, if we let $\mu_{x, \lambda}$ be the rescaled measure
\[
\mu_{x, \lambda}(A) = \lambda^{-n} \mu(x + \lambda A),
\]
then the family
\[
t \mapsto \mu_{x_0, \lambda}(\lambda^2 (t - t_0))
\]
will be a Brakke flow centered at $X_0 = (x_0, t_0)$, and parabolically dilated by $1/\lambda$, with free-boundary in $S/\lambda \subset U/\lambda$.
\end{remark}

\begin{remark}
If $\phi \in \cBT \cap C^2$, then the evolution equation \eqref{eqn:fb-evolution} can be written as
\[
\int \phi(\cdot, b) d\mu(b) - \int \phi(\cdot, a) d\mu(a) \leq \int_a^b \int - |S^T(H)|^2 \phi + (\partial_t - tr_V D^2 ) \phi d\mu(t) dt .
\]
This follows from Proposition \ref{prop:fb-finite-var}.
\end{remark}

Analogous to smooth flows, we will often work with Brakke flows as objects in spacetime.
\begin{definition}
Let $(\mu(t))_t$ be a (free-boundary) Brakke flow.  The \emph{spacetime support} is the closure (in spacetime) of
\[
\cup_t (\spt \mu(t) \times \{t\}) \subset \R^{N,1},
\]
taken over all times of definition.  The \emph{spacetime track} $\mcfM$ of $(\mu(t))_t$ is the spacetime support with associated multiplicities.

We shall often find it convenient to identify a (free-boundary) Brakke flow with its track $\mcfM$.  For example:
\begin{align*}
\mcfM_i \to \mcfM \quad &\text{means}\quad \mu_i(t) \to \mu(t) \quad \forall t \in I \\
\dilD_{1/\lambda}(\mcfM - X_0) \quad &\text{means}\quad \text{the flow dilated, translated in Remark \ref{rem:scaling}} \\
\int_{\mcfM(t)} f \quad &\text{means} \quad \int f d\mu(t) \\
\mcfM(t) \quad &\text{means}\quad \mu(t) .
\end{align*}
\end{definition}

When $S$ is a hyperplane, our definition reduces to the standard notion of Brakke flow.
\begin{prop}\label{prop:brakke-reflect}
Let $(\mu(t))_t$ be a Brakke flow with free-boundary in $P \subset \R^N$, for some hyperplane $P$, and write $A : \R^N \to \R^N$ for reflection about $P$.  Then the measures $\tilde \mu(t) = \mu(t) + A_\sharp \mu(t)$ define a Brakke flow in $\R^N$.
\end{prop}

\begin{proof}
Given any $\phi \in C^1_c(\R^N, \R_+)$, we have that $\phi + \phi \circ A \in \cBT(P)$.  Therefore, given any $(a, b)$ in the time interval of definition, we have
\begin{align*}
&\int \phi(\cdot, b) d\tilde\mu(b) - \int \phi(\cdot, a) d\tilde\mu(a) \\
&\leq \int_a^b \int -|S^T(H)|^2 \phi + S^T(H) \cdot D\phi + \partial_t \phi \,\, d\mu(t) \\
&\quad + \int_a^b \int -|S^T(H)\circ A|^2 \phi + (A \circ S^T(H) \circ A) \cdot D\phi + \partial_t \phi \,\, d(A_\sharp \mu(t)) \\
&\leq \int_a^b \int -|\tilde H|^2 \phi + \tilde H \cdot D\phi + \partial_t \phi d\tilde\mu(t) ,
\end{align*}
having used equations \eqref{eqn:reflect-varifold-var}, \eqref{eqn:reflect-varifold-L2}.
\end{proof}

\subsection{Mass bounds}

In choosing an appropriate cut-off we follow Buckland \cite{buckland}.
\begin{definition}\label{defn:mass-cutoff}
Take the \emph{cut-off radius} $\kappa$ to be any number $\leq r_S / c_0(n)$.  Let
\[
\eta(s) = (1-s)^4_+,
\]
and define the \emph{mass cut-off function}, at radius $\kappa$, to be
\begin{align*}
\phi_{S,\kappa}(x, t) &= \eta \left( (\kappa^2/\tau)^{3/4} \frac{|x|^2 - \alpha \tau}{\kappa^2} \right) \\
&= \left(1 - \kappa^{-2} (\kappa^2/\tau)^{3/4} (|x|^2 - \alpha \tau) \right)^4_+
\end{align*}
for some $\alpha = \alpha(n) \geq 1/2$ to be determined later.  Similarly, define the \emph{reflected} cutoff function
\[
\tilde \phi_{S, \kappa}(x, t) = \phi_{S, \kappa}(\tilde x, t).
\]
\end{definition}
There are a couple reasons for making these definitions.  First, $\phi + \tilde \phi \in \cT(S)$.  Second, the extra factor of $\tau^{-3/4}$ allows us to kill errors from the reflected $\tilde x$.  Third, we wish to keep the cutoff parabolic-scale-invariant

\begin{remark}\label{rem:cutoff-support}

If $\tau \leq \kappa^2 \left( (1 + \alpha) \cdot 20^2 \right)^{-4/3}$, then $\spt \phi_{S,\kappa} \subset B_{\kappa/20}(0)$.

If additionally $0 \in B_{\kappa/10}(S)$, then we have
\begin{equation}
\spt(\phi_{S,\kappa} + \tilde\phi_{S,\kappa}) \subset B_{\kappa/2}(0).
\end{equation}
This follows using Lemma \ref{lem:refl-bounds}, since we have for any $x \in \spt \phi_S$:
\begin{align*}
|\tilde x| 
&\leq |\tilde x - \refl(x)| + |\refl(x) - \refl(0)| + |\refl(0)| \\
&\leq \frac{c_1 (\kappa/5)^2}{r_S} + \kappa/20 + \kappa/10 \\
&\leq \kappa/2,
\end{align*}
where $\refl$ is the affine reflection about $T_{\zeta(0)}S$.  In particular,

Conversely, 
\begin{equation}
\phi_{S,\kappa} \geq 1/16 \text{ on } \{ x : |x|^2 \leq (\tau/\kappa^2)^{3/4} \kappa^2/2 \}.
\end{equation}
\end{remark}

We will use $\phi$ and $\tilde \phi$ as barriers.  We require certain conditions on $t$ and $x$ for $\tilde \phi$ to be an appropriate subsolution.
\begin{theorem}\label{thm:cutoff-evolution}
There is a $\beta_0(n)$ so that if $\kappa \leq r_S/c_1$, and $0 \in B_{\kappa/10}(S)$, then
\begin{equation}\label{eqn:cutoff-evolution}
(\partial_t - tr_L D^2) \tilde\phi_{S, \kappa}(x, t) \leq 0, \quad (\partial_t - tr_L D^2) \phi_{S, \kappa}(x, t) \leq 0
\end{equation}
for all $t \in (-\beta_0^2(n)\kappa^2, 0)$, and $x$ satisfying $d(x) \leq 10|\tilde x|$.

In particular, \eqref{eqn:cutoff-evolution} holds in the following cases:
\begin{enumerate}
\item[A)] $x \in \overline{\Omega}$; or

\item[B)] $0 \in S$ and $x$ arbitrary; or

\item[C)] $0 \in T_yS$ for some fixed $y \in S$, and $|x - y| \leq \min\{ |y|/10, r_S/c_1 \}$.
\end{enumerate}

\end{theorem}

\begin{proof}
First suppose $|d(x)| \leq 10 |\tilde x|$.  Then choose $\beta_0^2 = \left( (1 + \alpha) \cdot 20^2 \right)^{-4/3}$ as in Remark \ref{rem:cutoff-support}.  We calculate, using Lemma \ref{lem:refl-bounds}, 
\begin{align*}
(\partial_t - tr_L D^2) \tilde \phi_S
&= -\eta' (\kappa^2/\tau)^{3/4} \kappa^{-2} \left( -3/4 |\tilde x|^2/\tau - \alpha/4 + 2n \pm c(n)|\tilde x|/r_S \right) \\
&\leq |\eta'| (\kappa^2/\tau)^{3/2} \kappa^{-2} (-\alpha/4 + 2n \pm c(n)) \\
&\leq 0
\end{align*}
provided $\alpha(n)$ is sufficiently big.  The case of $\phi_S$ follows identically.

Let us demonstrate situations A)-C) imply the required estimate.  Suppose we are in case A) or B).  Let $y$ be the intersection of $S$ with the line segment connecting $0$ with $\tilde x$.  Then
\[
|\tilde x| = |\tilde x - y| + |y - 0| \geq d(\tilde x) + d(0) \geq d(x).
\]

Suppose we are in scenario C).  Then by Lemma \ref{lem:refl-bounds} we have
\begin{align*}
|\tilde x| \geq |\refl(x)| - \frac{2c_1}{r_S}|x - y|^2 \geq |x| - 2|x - y| \geq |x| - |y|/5, 
\end{align*}
And
\[
|y| \leq |x| + |x - y| \leq |x| + |y|/10.
\]
This proves $|d(x)| \leq |x| + |\tilde x| \leq 10 |\tilde x|$.
\end{proof}

From the above Theorem we deduce that mass at later times is controlled by mass at earlier times.
\begin{corollary}\label{cor:mass-bounds}
Let $(\mu(t))_{t \geq 0}$ be a Brakke flow with free-boundary in $S \subset U$, and $r_S/c_1 \geq \kappa$.

There is a constant $c(n)$ so that whenever $B_{R(t)}(z) \subset U$, where $R(t) := r + \kappa + c(n)t/\kappa$, then we have the estimate
\[
\mu(t)(B_r(z)) \leq c(n)^{1 + t/\kappa^2} \mu(0)(B_{R(t)}(z)).
\]
\end{corollary}

\begin{proof}
Since $\phi + \tilde \phi \subset \cBT(S, U, (-1, 0))$, we can plug $\phi + \tilde \phi$ into the evolution equation \eqref{eqn:fb-evolution}, then apply Theorem \ref{thm:cutoff-evolution} B) and Remark \ref{rem:cutoff-support} to deduce there is a $\gamma(n) := \beta_0^{3/4} \kappa/2$, so that
\[
\mu(t + s)(B_{\gamma \kappa}(y)) \leq c(n) \mu(t)(B_{\kappa}(y))
\]
for every $y \in S$, $t \geq 0$, and $s \in [0, \beta_0^2\kappa/2]$.

Choose a Vitali cover of $S \cap B_r(z)$ by balls of radius $\gamma \kappa$, then the balls of radius $\kappa$ will have overlap with multiplicity $\leq c(\gamma, n) = c(n)$.  Therefore, we deduce
\begin{equation}\label{eqn:S-ball-est}
\mu(t + s)(B_{\gamma \kappa/2}(S) \cap B_r(z)) \leq c(n) \mu(t)(B_{\kappa}(S) \cap B_{r+\kappa}(z)),
\end{equation}
for $t$ and $s$ as before.

The interior mass bound is similar.  Define the barrier function
\[
\psi(x, t) = \left( 1 - (\gamma \kappa/8)^{-2}(|x|^2 - 2n\tau) \right)^4_+.
\]
One can check directly that $(\partial_t - tr_V D^2)\psi \leq 0$, and
\[
\psi(x, 0) \geq 1/16 \quad \text{ on } B_{\gamma\kappa/16}(0),
\]
and
\[
\spt \psi \subset \subset B_{\gamma \kappa/2}(0) \quad \text{and} \quad |\psi| \leq 2 \quad \text{when } \tau \in [0, \frac{\gamma^2\kappa^2}{20n}].
\]

So $\psi$ is admissible when $d(0, S) \geq \gamma\kappa/2$ and $\tau$ is restricted as above.  We therefore have
\begin{equation}\label{eqn:inside-ball-est}
\mu(t + s)(B_{\gamma\kappa/16}(y)) \leq c \mu(t)(B_{\gamma\kappa/2}(y)) 
\end{equation}
for every $y \not\in B_{\gamma\kappa/2}(S)$, and $t \geq 0$, and $s \in [0, \frac{\gamma^2\kappa^2}{20n}]$.

Apply \eqref{eqn:inside-ball-est} to a Vitali cover of $B_r(z) \setminus B_{\gamma\kappa/2}(S)$ like above, and use \eqref{eqn:S-ball-est}, to deduce
\[
\mu(t + s)(B_r(z)) \leq c(n) \mu(t) (B_{r+\kappa}(z)) \quad \forall s \in [0, \kappa^2/c(n)].
\]
for some fixed constant $c(n)$.  Now iterate this in $t$ to obtain the required estimate. 
\end{proof}

\begin{corollary}
Let $(\mu(t))_{t \geq 0}$ be a Brakke flow with free-boundary in $S \subset \R^N$, and $r_S/c_1 \geq \kappa$.  If $\spt \mu(0) \subset B_R(0)$, then
\[
\spt\mu(t) \subset B_{R + \kappa + c(n)t/\kappa}(0) \quad \forall t  .
\]

Hence, $\mu(t)$ can only move a finite distance in finite time.
\end{corollary}

\begin{proof}
Apply Corollary \ref{cor:mass-bounds} to balls outside the support of $\mu$.
\end{proof}

\subsection{Compactness}

The following semi-decreasing property is crucial for (free-boundary) Brakke flows.
\begin{prop}[compare Theorem 7.2 in Ilmanen \cite{ilmanen:elliptic-reg}]\label{thm:semi-decreasing}
Let $(\mu(t))_{t \geq -1}$ be a Brakke flow with free-boundary in $S \subset U$, satisfying
\[
\mu(t)(K) \leq C_1(K) \quad \forall t \geq -1
\]
for every compact $K \subset U$.  Then the following holds:

\begin{enumerate}
\item[A)] Take $\phi \in \cBT(S, U) \cap C^2$ independent of time.  Then there is a constant $C_\phi(\phi, C_1(\spt \phi))$ so that
\[
t \mapsto \mu(t)(\phi) - C_\phi t
\]
is decreasing in $t$.

\item[B)] The left-/right-limits of $\mu(t)$ exist at every $s \geq -1$, and satisfy
\[
\lim_{t \to s_-} \mu(t) \geq \mu(s) \geq \lim_{t \to s_+} \mu(t).
\]

\item[C)] For every $K \subset\subset U$, there is a compact $\tilde K$ (depending only on $K$, $U$, $|b - a|$) and constant $D_1(K) = D_1(C_1(\tilde K), |b - a|)$, so that $K \subset \tilde K \subset U$ and
\[
\int_a^b \int_K |S^T(H)|^2 d\mu(t) dt \leq D_1(K) .
\]
\end{enumerate}
\end{prop}

\begin{proof}
Given any $b > a \geq -1$, we have
\begin{align*}
\int \phi d\mu(b) - \int \phi d\mu(a) 
&\leq \int_a^b \int -|S^T(H)|^2\phi + S^T(H) \cdot D\phi d\mu(t) dt \\
&\leq \int_a^b \int -\frac{1}{2} |S^T(H)|^2 + 2\frac{|D\phi|^2}{\phi} d\mu(t) dt \\
&\leq -\frac{1}{2} \int_a^b |S^T(H)|^2 + C(\phi) C_1(\spt \phi) |d - c| .
\end{align*}
The last bound follows from \cite{ilmanen:elliptic-reg} Lemma 6.6.  This shows parts A) and C), choosing $\phi$ to be $\equiv 1$ on $K$.  Part B) follows by applying part A) to every $\phi$ in a dense subset of $C^0_C(U, \R_+)$, and using the compactness of Radon measures.
\end{proof}

The semi-continuity of Proposition \ref{thm:varifold-compactness} allows us to prove a compactness Theorem for free-boundary Brakke flows.
\begin{theorem}\label{thm:brakke-compactness}
Let $(\mu_i(t))_{t \geq -1}$ be a sequence of Brakke flows with free-boundary in $S_i \subset U_i$.  Let $\inf r_{S_i} > 0$, $U_i \to U$, and $S_i \to S$ in $C^3_{loc}$.  Suppose $\mu_i(t)$ satisfy
\[
\sup_i \sup_t \mu_i(t)(K) \leq C_1(K) < \infty \quad \forall \text{ compact } K \subset U
\]

Then there is a $n$-Brakke flow $(\mu(t))_{t \geq -1}$ with free-boundary in $S \subset U$, and a subsequence $i'$, so that
\[
\mu_{i'}(t) \to \mu(t)
\]
as Radon measures for every $t\geq -1$.  For a.e. $t \geq -1$, there is a further subsequence $i''$ (depending on $t$), so that
\[
V_{i''}(t) \to V(t)
\]
as varifolds.  Here $V_{i''}(t)$ and $V(t)$ are the integral $n$-varifolds, with free-boundary in $S_{i''}$, $S$, associated to $\mu_{i''}(t)$, $\mu(t)$ (resp.).
\end{theorem}

\begin{proof}
Let $\mathcal C$ be a countable subset of $\{ \phi \in C^2_c(U, \R_+) : D\phi \in \cT(S) \}$, which is dense in $C^0_c(U, \R_+)$, see Proposition \ref{prop:fb-density-theorems}.  For any $\phi \in \mathcal C$, there is a constant $C_\phi = C(\phi, C_1(\spt\phi))$ so that
\[
L_{\phi, i}(t) = \mu_i(t)(\phi) - C_\phi t
\]
is decreasing in $t$.  We can assume $L_{\phi, i}(0) \leq C_\phi$ also.

By Helly's selection principle and diagonalization we can pass to a subsequence, also denoted $i$, so that for each $\phi \in \mathcal C$ there is a decreasing function $L_\phi(t)$ satisfying
\[
L_{\phi, i}(t) \to L_\phi(t) \quad \forall t \geq -1.
\]

In other words, $\lim_i \mu_i(t)(\phi)$ exists every $t \geq -1$, and $\phi \in \mathcal C$.  Since $\mathcal C$ is dense in $C^0_c(U, \R_+)$, the usual compactness of Radon measures implies there exist a collection of Radon measures $(\mu(t))_{t \geq -1}$ so that
\[
\mu_i(t) \to \mu(t)
\]
as Radon measures for every $t \geq -1$.  Moreover, since $L_\phi(t) \equiv \mu(t)(\phi)$ is decreasing for each $\phi \in \mathcal C$, by the same arguments of Proposition \ref{thm:semi-decreasing} we have that $t \mapsto \mu(t)$ is continuous at a set of times of full measure.

We show $(\mu(t))_t$ is a Brakke flow with free-boundary.  Fix an $-1 \leq a < b$.  If we let
\[
f_K(t) = \liminf_i \int_K 1 + |S^T(H_i)|^2 d\mu_i(t), 
\]
Then by Proposition \ref{thm:semi-decreasing} and Fatou's lemma we have that for a.e. $t \in (a, b)$, $f_K(t) < \infty$ for every compact $K \subset U$.

For a.e. $t \geq -1$, we can pass to a subsequence $i'$ (depending on $t$) so that the following is satisfied:
\begin{enumerate}
\item[A)] $\mu_{i'}(t) = \mu_{V_{i'}}$ for some integral varifold $V_{i'}$, having free-boundary in $S_{i'}$;

\item[B)] $\sup_{i'} \int_K 1 + |S^T(H_{i'})|^2 d\mu_{i'}(t) < \infty$ for every compact $K$.
\end{enumerate}

By Theorem \ref{thm:varifold-compactness}, we can pass to a further subsequence, to obtain convergence $V_{i'} \to V$ to some integral varifold, with free-boundary in $S \subset U$, and $S^T(H_V) \in L^2_{loc}(U, \R^N; \mu_V)$.  Since $\mu_i(t) \to \mu(t) = \mu_V$ independent of $i'$, we see $V$ is determined independently of sequence $i'$ also.

Take $\phi \in \cBT(S, U, [a, b])$, for some fixed $-1 \leq a < b < \infty$.  Let $B \subset (-1, \infty)$ be the set of times at which $t \mapsto \mu(t)$ is continuous, and at which $\mu(t) = \mu_{V(t)}$ as above.  From the previous paragraphs, $B$ has full measure.  Given any $t_i \to t$, with $t_i, t \in B \cap [a, b]$, we have by definition of $B$ that $\mu(t_i) \to \mu(t)$, and hence the associated varifolds $V(t_i) \to V(t)$ also.

Using \eqref{eqn:H-lsc}, and convergence $\mu(t_i) \to \mu(t)$, we deduce
\[
\int -|S^T(H)|^2 \phi + S^T(H) \cdot D\phi + \partial_t \phi d\mu(t) \geq \limsup_i \int -|S^T(H)|^2 \phi + S^T(H) \cdot D\phi + \partial_t \phi d\mu(t_i)
\]
Therefore
\[
t \mapsto \int -|S^T(H)|^2 \phi + S^T(H) \cdot \phi + \partial_t \phi d\mu(t)
\]
is upper-semi-continuous on $B \cap [a, b]$, and hence measurable on $[a, b]$.

We show $\mu(t)$ satisfies the inequality \eqref{eqn:fb-evolution}.  Fix $\phi \in \cBT(S, U, [a, b]) \cap C^2$.  We can choose a sequence $\phi_i \in \cBT(S_i, U_i, [a, b]) \cap C^2$ so that $\phi_i \to \phi$ in $C^2$, and $\spt \phi_i(\cdot, t) \subset \spt \phi(\cdot, t)$ for every $t$.  Since the $\phi_i$, $\phi$ are uniformly bounded in $C^2$, we can find a fixed function $\psi \in C^0_c(U, \R_+)$, so that
\[
\frac{|D\phi|^2}{\phi} + \frac{|D\phi_i|^2}{\phi_i} \leq \psi \quad \forall i.
\]
And therefore, for each $i$ we have
\[
|S^T(H_i)|^2\phi_i - H_i \cdot D\phi_i + \psi \geq 0.
\]

For each $t$, we have by construction that
\[
\int \partial_t \phi_i d\mu_i(t) \to \int \partial_t \phi d\mu(t), \quad \int \psi d\mu_i(t) \to \int \psi d\mu(t),
\]
and clearly each term above is uniformly bounded in $t$.

Using the dominated convergence theorem, and Fatou's lemma, we have
\begin{align}
&\int \phi d\mu(a) - \int \phi d\mu(b) + \int_a^b \int \psi d\mu(t) dt \nonumber\\
&= \lim_i \left( \int \phi_i d\mu_i(a) - \int \phi_i d\mu_i(b) + \int_a^b \int \psi d\mu_i(t) dt \right) \nonumber\\
&\geq \liminf_i \left( \int_a^b \int |S^T(H_i)|^2 - H_i \cdot D\phi_i + \psi  d\mu_i(t) dt - \int_a^b \int \partial_t \phi_i d\mu_i(t) dt \right) \nonumber\\
&\geq \int_a^b \left( \liminf_i \int |S^T(H_i)|^2 - H_i \cdot D\phi_i + \psi d\mu_i(t) \right) dt - \int_a^b \int \partial_t \phi d\mu(t) dt \nonumber\\
&=: \int_a^b M_\phi(t)  dt - \int_a^b \int\partial_t \phi d\mu(t) dt \label{eqn:brakke-ineq-compact}
\end{align}

As before, for a.e. $t \in (a, b)$, we can choose a subsequence $i'$ (depending on $t$), so that
\[
M_\phi(t) = \lim_{i'} \int |S^T(H_{i'})|^2 \phi_{i'} - H_{i'} \cdot D\phi_{i'} + \psi d\mu_{i'}(t),
\]
and $V_{i'}(t) \to V(t)$ as free-boundary varifolds.  Semi-continuity \eqref{eqn:H-lsc}, and convergence $\mu_i(t) \to \mu(t)$, then implies
\[
M_\phi(t) \geq \int |S^T(H_V)|^2 \phi - H_V \cdot D\phi + \psi d\mu_V ,
\]
independently of subsequence $i'$.

Plugging this back into \eqref{eqn:brakke-ineq-compact}, we find that $\mu(t)$ satisfies inequality \eqref{eqn:fb-evolution} for every $\phi \in \cBT(S, U, [a, b]) \cap C^2$.  Since $\cBT(S, U, [a, b]) \cap C^2$ is dense in $\cBT(S, U, [a, b])$, we deduce $(\mu(t))_{t \geq -1}$ is a Brakke flow with free-boundary in $S \subset U$.
\end{proof}

\section{Monotonicity}

We prove a monotonicity for the following reflected and truncated Gaussian.
\begin{definition}
Define the \emph{Gaussian heat kernel}
\[
\rho_S(x, t) = (4\pi\tau)^{-n/2}e^{-\frac{|x|^2}{4\tau}},
\]
and define the relfected heat kernel $\tilde \rho_S(x, t) = \rho_S(\tilde x, t)$.  Recall $\tau \equiv -t$.

Define the \emph{reflected, truncated head kernel} to be
\begin{equation}
f_{S, \kappa} = \rho_S \phi_{S, \kappa} + \tilde \rho_S \tilde \phi_{S, \kappa},
\end{equation}
where $\phi_{S,\kappa}$, $\tilde \phi_{S,\kappa}$ as in Definition \ref{defn:mass-cutoff}.  Notice that by construction $f \in \cT(S)$.

The above definition gives us the appropriate Gaussian density centered at the spacetime origin.  Given an $X_0 = (x_0, t_0)$, define the \emph{recentered} reflected, truncated heat kernel is the function
\begin{align*}
f_{S, \kappa, (x_0, t_0)}(x, t) &= \phi(x - x_0, t - t_0) \rho(x - x_0, t - t_0) \\
& \quad + \phi(\tilde x - x_0, t - t_0) \rho(\tilde x - x_0, t - t_0),
\end{align*}
so that $f_{S, \kappa, (0, 0)} \equiv f_{S, \kappa}$.
\end{definition}

\begin{theorem}\label{thm:monotonicity}
Let $(\mu(t))_{t \geq -1}$ be a free-boundary Brakke flow supported in $\Omega \subset U$ (so, $S = \partial\Omega$).  Suppose $0 \in B_{\kappa/10}(S) \cap \bar \Omega$, and $d(0, \partial U) \geq \kappa$.

There are $\tau_0(\kappa, n)$, $A(\kappa, n)$ so that if $\kappa \leq r_S/c_1$, then
\[
t \mapsto e^{A(-t)^{1/4}} \int f_{S, \kappa} d\mu(t) + A M (-t)
\]
is decreasing in $t \in [-\tau_0, 0]$.  Here $M$ is a any constant bounding
\[
M \geq \mu(-\tau_0)(\phi_{S, \kappa}(\cdot, -\tau_0) + \tilde \phi_{S, \kappa}(\cdot, -\tau_0)) .
\]
\end{theorem}

\begin{proof}
In the following we write $g = O(f)$ to mean $|g| \leq c(n, \kappa) |f|$.  For $\mu_V$-a.e. $x$, write
\begin{align*}
\pi^T, \pi^\perp \quad &\text{for the linear pojections onto}\quad T_x V, (T_x V)^\perp, \\
\tilde \pi^T, \tilde \pi^\perp \quad &\text{for the linear pojections onto the reflected spaces} \quad \widetilde{T_x V}, \widetilde{(T_x V)^\perp}.
\end{align*}
Since both $0$ and $\spt\mcfM$ lie to one side of the barrier $S = \partial \Omega$, we have $d(x) = O(|\tilde x|)$.

Let us pick an ON basis $e_i$ of $T_x V$.  By direct computation we have
\[
D_i \tilde \rho = \frac{-<\tilde e_i, \tilde x>}{2\tau} \tilde \rho + O(\frac{|\tilde x|^2}{\tau}) \tilde \rho,
\]
and therefore
\[
S^T(H_V) \cdot D \tilde\rho = -\frac{<\widetilde{S^T(H_V)}, \tilde x>}{2\tau} \tilde\rho + O(\frac{|\tilde x|^2}{\tau} |S^T(H_V)|) \tilde\rho.
\]
We used the trivial relation $<e_i, Y> = <\tilde e_i, \widetilde{Y}>$.

By direct computation, we have
\[
\sum_i (D_i|\tilde x|^2)^2 = 4 |\tilde \pi^T(\tilde x)|^2 + O(|\tilde x|^2 d) , \quad \sum_i D_i D_i |\tilde x|^2 = 2n + O(|\tilde x|) .
\]

Therefore we have
\[
(\partial_t + tr D^2)\tilde\rho = -\frac{|\tilde \pi^\perp(\tilde x)|^2}{4\tau^2} \tilde \rho + O(|\tilde x|/\tau + |\tilde x|^3/\tau^2)\tilde\rho .
\]

By a result of Brakke \cite{brakke}, $H_V = \pi^\perp(H_V)$ $\mu_V$-almost everywhere.  Since $T_x V \subset T_x S$ at $\mu_V$-a.e. $x \in S$, we deduce
\[
S^T(H_V) = \pi^\perp(S^T(H_V)) \quad \text{$\mu_V$-a.e. $x$}.
\]
Using this and the above calculations, we have at $\mu_V$-a.e. $x$:
\begin{align*}
&(\partial_t + 2 S^T(H_V) \cdot D + tr_V D^2) \tilde\rho  \\
&= \left[ -\frac{|\tilde \pi^\perp(\tilde x)|^2}{4\tau^2} - \frac{<\tilde \pi^\perp(\tilde x), \widetilde{S^T(H_V)}>}{\tau} \right] \tilde \rho + O(|\tilde x|/\tau + |\tilde x|^3/\tau^2 + |S^T(H)| |\tilde x|^2/\tau) \tilde\rho \\
&\leq \left[ - \left| \frac{\tilde \pi^\perp(\tilde x)}{2\tau} + \widetilde{S^T(H)} \right|^2 + |S^T(H)|^2 \right] \tilde \rho + O(|\tilde x|/\tau + |\tilde x|^3/\tau^2 + |S^T(H)| |\tilde x|^2/\tau) \tilde\rho .
\end{align*}

Without much rigmarole we have also
\[
(\partial_t + 2 S^T(H_V) \cdot D + tr_V D^2) \rho = \left[ - \left| \frac{\pi^\perp(x)}{2\tau} + S^T(H)\right|^2 + |S^T(H)|^2 \right] \rho .
\]

Since both $\rho + \tilde \rho$ and $\phi + \tilde \phi$ lie in $\cT(S)$, we have
\begin{align*}
&\int \phi tr_V D^2 \rho - \rho tr_V D^2 \phi + \tilde \phi tr_V D^2\tilde\rho - \tilde\rho tr_V D^2 \tilde \phi d\mu_V \\
&\quad= \int - S^T(H) \cdot D f + 2 S^T(H) \cdot (\phi D \rho + \tilde\phi D \tilde\rho) d\mu_V
\end{align*}

We now calculate, using Theorem \ref{thm:cutoff-evolution}, and ensuring $\tau_0 \leq \beta_0(n) \kappa^2$, 
\begin{align*}
\overline{D_t} \int f \, d\mu(t) 
&\leq \int S^T(H) \cdot Df + \partial_t f - |S^T(H)|^2 f \,d\mu(t) \\
&\leq \int S^T(H) \cdot Df - 2 S^T(H) \cdot (\phi D \rho + \tilde\phi D\tilde \rho)  \, d\mu(t) \\
&\quad + \int (tr_V D^2 \tilde\phi)\tilde\rho - \phi(tr_V D^2\tilde\rho) + (tr_V D^2\phi) \rho - \phi(tr_V D^2\rho) \,d\mu(t)\\
&\quad - \int \left| \frac{\tilde \pi^\perp(\tilde x)}{2\tau} + \widetilde{S^T(H)} \right|^2 \tilde\phi\tilde\rho + \left| \frac{\pi^\perp(x)}{2\tau} + S^T(H) \right|^2 \phi\rho \,d\mu(t) \\
&\quad + C \int \tilde\phi \tilde\rho (|\tilde x|/\tau + |\tilde x|^3/\tau^2 + |S^T(H)| |\tilde x|^2/\tau)  \,d\mu(t).
\end{align*}
Here $C = C(\kappa, n)$, and $\overline{D_t}$ represents the upper-derivative in the sense of $\limsup$s of difference quotients.

We do some subcalculations.  First,
\begin{align*}
&C \frac{|\tilde x|^2}{\tau} |\widetilde{S^T(H)}| - \left|\frac{\tilde \pi^\perp(\tilde x)}{2\tau} + \widetilde{S^T(H)} \right|^2 \\
&\leq C \frac{|\tilde x|^2}{\tau} \left( \left|\frac{\tilde \pi^\perp(\tilde x)}{2\tau} + \widetilde{S^T(H)} \right| + \frac{|\tilde x|}{2\tau} \right) - \left| \frac{\tilde \pi^\perp (\tilde x)}{2\tau} + \widetilde{S^T(H)} \right|^2 \\
&\leq C\frac{|\tilde x|^3}{\tau^2} + C^2 \frac{|\tilde x|^4}{\tau^2} .
\end{align*}

Second, setting $\alpha := 1 + 1/(2+2n) > 1$ we have
\begin{align*}
\tilde \rho \frac{|\tilde x|}{\tau}
&\leq 1 + \left( \tilde\rho \frac{|\tilde x|}{\tau} \right)^\alpha \\
&= 1 + c \tau^{\frac{n\alpha}{2} - \frac{\alpha}{2}} \left( e^{-\frac{|\tilde x|^2}{4\tau}} \frac{|\tilde x|}{2\sqrt{\tau}} \right)^\alpha \\
&\leq 1 + c(n) \tau^{-\frac{n\alpha}{2} - \frac{\alpha}{2} + \frac{n}{2}} \tilde \rho \\
&= 1 + c(n) \tau^{-3/4} \tilde \rho .
\end{align*}
We used that, for any $\beta, \gamma > 0$, that $y^\beta e^{-y} \leq C(\beta, \gamma) (e^{-y})^{1-\gamma}$.  Precisely the same calculation holds for $|\tilde x|^3/\tau^2$, since the relative powers again differ by $1/2$.

We put the three calculations together, to deduce:
\begin{align*}
\overline{D_t} \int f  \, d\mu(t)
&\leq C \int \tilde\phi \tilde\rho (\frac{|\tilde x|}{\tau} + \frac{|\tilde x|^3}{\tau^2} )  \,d\mu(t)  \\
&\leq C \tau^{-3/4} \int f + C \int \phi + \tilde\phi  \, d\mu(t)  \\
&\leq C\tau^{-3/4} \int f + C \int \phi + \tilde \phi \, d\mu(-\tau_0) \\
&\leq A \tau^{-3/4} \int f + A M . \qedhere
\end{align*}
\end{proof}

Using the above we define a Gaussian density in a neighborhood of the barrier.  For points outside this neighborhood we can use the standard truncated Gaussian density, and by our one-sidedness assumption these will be compatible across the transition region.
\begin{definition}
Let $\mcfM \equiv (\mu(t))_{t \geq -1}$ be a free-boundary Brakke flow supported in $\Omega \subset U$.  For $t_0 > -1$, and $x_0 \in \overline{\Omega}$, define the \emph{reflected Gaussian density} of $\mcfM$ at $X_0 = (x_0, t_0)$ as follows.

For $\kappa$, $r$ satisfying:
\[
\kappa \leq \min\{ r_S/c_1, d(x_0, \partial U)\}, \quad r^2 \leq \min\{ \tau_0(n, \kappa), t_0 + 1 \} ,
\]
we set
\begin{align*}
&\Theta_{refl(S, \kappa)}(\mcfM, X_0, r) \\
&= \left\{ \begin{array}{l l} \int f_{S, \kappa, X_0}(x, t_0 - r^2) d\mu(t_0 - r^2)(x) & x_0 \in B_{\kappa/10}(S) \cap \bar\Omega \\
\int \phi(x - x_0, -r^2) \rho(x - x_0, -r^2) d\mu(t_0 - r^2)(x) & x_0 \in \bar\Omega \setminus B_{\kappa/10}(S) . \end{array} \right. .
\end{align*}
\end{definition}

\begin{remark}
By Remark \ref{rem:cutoff-support}, and our one-sidedness assumption, the two cases agree near $\partial B_{\kappa/10}(S) \cap \bar\Omega$ provided $\tau_0$ is sufficiently small (depending only on $n, \kappa$).
\end{remark}

\begin{remark}
$\Theta_{refl}$ is parabolic scale-invariant, in the sense that
\begin{equation}
\Theta_{refl(S/\lambda, \kappa/\lambda)}(\dilD_{1/\lambda} \mcfM, 0, R) = \Theta_{refl(S, \kappa)}(\mcfM, 0, \lambda R) .
\end{equation}
\end{remark}

\begin{remark}
If $S$ is a plane, then $\Theta_{refl(S,\kappa)}$ is the a truncated Gaussian density of the reflected flow.
\end{remark}

In this new notation, Theorem \ref{thm:monotonicity} implies
\begin{theorem}\label{thm:pretty-mono}
Let $\mcfM \equiv (\mu(t))_{t \geq -1}$ be a free-boundary Brakke flow supported in $\Omega \subset U$.  For any $x_0 \in \bar\Omega$, $t_0 > -1$, and $\kappa \leq \min\{r_S/c_1, d(x_0, \partial U)\}$, we have that
\[
r \mapsto  e^{A\sqrt{r}} \Theta_{refl(S, \kappa)}(\mcfM, X_0, r) + AM r^2
\]
is increasing in $r \in [0, \min(\sqrt{\tau_0}, \sqrt{1 + t_0})]$.  Here $M$ is any constant bounding
\[
M \geq \mu(\min(t_0 - \tau_0, -1))(B_{\kappa/2}(x_0)).
\]

In particular, the limit
\[
\Theta_{refl(S, \kappa)}(\mcfM, X_0) := \lim_{r \to 0} \Theta_{refl(S,\kappa)}(\mcfM, X_0, r)
\]
exists and is finite, for every $t_0 > -1$.
\end{theorem}

\begin{proof}
For $x_0$ near the barrier this is immediate from Theorem \ref{thm:monotonicity}.  Away from the barrier, $\Theta$ is actually monotone in $r$ without error terms, by the computations of Theorem \ref{thm:monotonicity}.
\end{proof}

\begin{remark}
In the following section we will show this limit is independent of (admissible) choice of $\kappa$.
\end{remark}

\section{Tangent flows}

We prove the existence of tangent flows, as self-similar Brakke flows with free-boundary in a plane, and show that the Gaussian density of the reflected tangent flow agrees with the original reflected density at the point.  Throughout this section we take $\mcfM \equiv (\mu(t))_{t \geq -1}$ to be a free-boundary Brakke flow supported in $\Omega \subset U$, and write (as usual) $S = \partial\Omega$.

Let us recall some standard definitions.  The usual \emph{Gaussian density} of a $n$-Brakke flow $\mcfM$, at a point $X_0 = (x_0, t_0)$ and scale $r$, is defined by
\[
\Theta(\mcfM, (x_0, t_0), r) := (4\pi r^2)^{-n/2} \int_{\mcfM(t_0 - r^2)} e^{-\frac{|x - x_0|^2}{4r^2}}.
\]
As proven in \cite{huisken:monotonicity}, $\Theta(\mcfM, X, r)$ is increasing in $r$, and strictly increasing unless $\mcfM$ is a self-shrinker (parabolic cone) centered at $X$.  The \emph{Guassian density at $X$} is the limit
\[
\Theta(\mcfM, X) := \lim_{r \to 0} \Theta(\mcfM, X, r).
\]
If $\mcfM$ is an ancient Brakke flow (so, defined for all negative time), the \emph{Gaussian density at $\infty$} is defined to by
\[
\Theta(\mcfM) := \lim_{r \to \infty} \Theta(\mcfM, 0, r).
\]
One can check directly that if $\Theta(\mcfM) < \infty$, then
\[
\Theta(\mcfM, (x, t), r) \leq \Theta(\mcfM) \quad \forall t \leq 0, \quad \forall r \geq 0.
\]

The \emph{Euclidean density} of an $n$-varifold $V$, at a point $x$ and scale $r$, is defined to be
\[
\Theta_{eucl}(V, x, r) := \frac{\mu_V(B_r(x))}{r^n}
\]
When $V$ is stationary, $\Theta_{eucl}$ is increasing in $r$, and we can define the Euclidean density at a point $x$, or at $\infty$ (resp.) by 
\[
\Theta_{eucl}(V, x) := \lim_{r \to 0} \Theta_{eucl}(V, x, r), \quad \Theta_{eucl}(V) := \lim_{r \to \infty} \Theta_{eucl}(V, 0, r).
\]

We first show that mass of $\mcfM$ is controlled in the dilates.
\begin{lemma}
For any $r > 0$, and $\tau \in (0, \infty)$, then provided $\lambda$ is sufficiently small, we have
\[
(\dilD_{1/\lambda} \mcfM) (-\tau) (B_r(0)) \leq C(r, \tau, \kappa, n) \Theta_{refl(S, \kappa)}(\mcfM, 0), 
\]
provided $\kappa \leq \min\{r_S/c_1, d(0, \partial U)\}$.
\end{lemma}

\begin{proof}
By Remark \ref{rem:cutoff-support} for $\lambda$ small we have $\phi_{S/\lambda, \kappa/\lambda}(x, -\tau) \geq 1/16$ on $B_r(0)$.  

Therefore, provided $\lambda \leq \lambda_0(\mcfM, \kappa, n)$, we have
\begin{align*}
2 \Theta_{refl(S,\kappa)}(\mcfM, 0)
&\geq \Theta_{refl(S/\lambda, \kappa/\lambda)}(\dilD_{1/\lambda} \mcfM, 0, \sqrt{\tau}) \\
&\geq \frac{1}{16} (4\pi \tau)^{-n/2} \int_{(\dilD_{1/\lambda} \mcfM)(-\tau) \cap B_r} e^{-|x|^2/4\tau} \\
&\geq \frac{1}{c(r, \tau, n)} ||(\dilD_{1/\lambda}\mcfM)(-\tau)||(B_r). \qedhere
\end{align*}
\end{proof}

Since $(S - x)/\lambda$ converges in $C^3_{loc}$ to a plane as $\lambda \to 0$, whenver $x \in S$, we have by Theorem \ref{thm:brakke-compactness}:
\begin{prop}\label{prop:tangent-flows}
Let $x_0 \in \overline{\Omega}$, $t_0 > -1$, and $\lambda_i \to 0$.  Write $X_0 = (x_0, t_0)$.

After passing to a subsequence, there is an ancient $\mcfM'$ so that
\[
\dilD_{1/\lambda_i}(\mcfM - X_0) \to \mcfM'.
\]
Here $\mcfM'$ is either a Brakke flow in $\R^N$ (if $x_0 \not\in S$), or a free-boundary Brakke flow in some half-space in $\R^N$ (if $x_0 \in S$).

If $\mcfM'$ has free-boundary it can be reflected to obtain a Brakke flow $\tilde\mcfM'$ in $\R^N$ (without boundary).  Otherwise simply let $\tilde\mcfM' = \mcfM'$.
\end{prop}

We show that $\mcfM'$ moves by self-shrinking.
\begin{lemma}\label{lem:mass-decay}
For any $\Lambda$, and $\tau \in (0, \tau_0/2)$, and $\kappa \leq \min\{r_S/c_1, d(0, \partial U)\}$, we have
\[
\int_{\mcfM(-\tau) \cap \{ |x|^2 \geq \Lambda \tau\} } f(x, -\tau)   \leq c(M, \kappa, n) e^{-\Lambda/8}
\]
where $M$ is any bound on $\mu(-\tau_0/2)(B_{\kappa/2}(0))$, and $r_S \geq \kappa$.  In particular, $M$ is independent of $\tau$ and $\Lambda$.
\end{lemma}

\begin{proof}
First observe that
\begin{align*}
\phi_{S, \kappa}(x, -\tau) \leq \left(1 - \kappa^{-2+3/2} |x|^2/(2\tau)^{3/4} + \alpha \kappa^{-2+3/2} (2\tau)^{1-3/4} \right)^4_+ = \phi_{S, \kappa}(x, -2\tau) .
\end{align*}

Therefore provided $\tau \leq \tau_0/2$ we can use the monotonicity formula centered at $(0, \tau)$ to deduce
\begin{align*}
&\int_{\mcfM(-\tau) \cap \{ |x|^2 \geq \Lambda \tau \}} f_{S, \kappa}(x, -\tau) \\
&\leq e^{-\Lambda/8} 2^{m/2} \int_{\mcfM(-\tau)} f_{S, \kappa, (0, \tau)}(x, -\tau)  \\
&\leq e^{-\Lambda/8}2^{m/2} \left[ e^{A\tau_0^{1/4}} \int_{\mcfM(-\tau_0/2)} f_{S, \kappa, (0, \tau)}(x, -\tau_0/2) + AM\tau_0 \right] \\
&\leq e^{-\Lambda/8} C(\kappa, n, M) .\qedhere
\end{align*}
\end{proof}

\begin{theorem}\label{thm:tangent-flows}
The (reflected) Brakke flow $\tilde\mcfM'$ from Proposition \ref{prop:tangent-flows} is an ancient self-shrinker, having density
\[
\Theta(\tilde\mcfM') = \Theta_{refl(S, \kappa)}(\mcfM, X_0)
\]
for any admissible $\kappa$.
\end{theorem}

\begin{corollary}
The pointwise density $\Theta_{refl(S, \kappa)}(\mcfM, X)$ is independent of $\kappa$.
\end{corollary}

\begin{proof}[Proof of Theorem]
We can suppose $X_0 = 0$.  Take $\mcfM_i = \dilD_{1/\lambda_{i}}\mcfM \to \mcfM'$ as above, and $\kappa \leq \min\{r_S/c_1, d(0, \partial U)\}$.  Let $\refl$ be reflection about the limit barrier plane if it exists, and formally $\infty$ if it does not.  By the dominated convergence theorem we have
\begin{align*}
\int_{\mcfM'(-\tau_0) \cap B_R} \tau_0^{-m/2} \left( e^{-|x|^2/4\tau_0} + e^{-|\refl(x)|^2/4\tau_0} \right) 
&= \lim_i \int_{\mcfM_i(-\tau_0) \cap B_R} f_{S/\lambda_i, \kappa/\lambda_i}(x, -\tau_0) \\
&= \lim_i \int_{\mcfM(-\lambda_i^2 \tau_0) \cap B_{\lambda_i R}} f_{S, \kappa}(x, -\lambda_i^2 \tau_0)
\end{align*}

Writing $\tau \equiv \lambda_i^2 \tau_0$, we have
\begin{align*}
B_{\lambda_i R} = \{ x : |x| \leq \lambda_i R = \frac{R}{\sqrt{\tau_0}} \sqrt{\tau} \} = \{ x : |x|^2 \leq \frac{R^2}{\tau_0} \tau \} .
\end{align*}
So every term in the limiting sequence is of the form
\[
\int_{\mcfM(-\tau) \cap \{ |x|^2 \leq \frac{R^2}{\tau_0} \tau\}} f_{S, \kappa}(x, -\tau)
\]

By the Lemma \ref{lem:mass-decay} we have
\[
\int_{\mcfM'(-\tau_0) \cap B_R} \tau_0^{-m/2} \left( e^{-|x|^2/4\tau_0} + e^{-|\refl(x)|^2/4\tau_0} \right)  = \Theta_{refl(S)}(\mcfM, 0) \pm \epsilon
\]
where $\epsilon \to 0$ as $R \to \infty$.  By the monotone convergence theorem the LHS converges to the standard Gaussian density of the reflected Brakke flow $\tilde\mcfM'$.  The above equality holds for every $\tau_0$, so $\tilde\mcfM'$ has constant density.

By the standard monotonicity formula (\cite{huisken:monotonicity}) we deduce $\tilde\mcfM'$ is a self-shrinker.
\end{proof}

The above Theorem and Corollary motivate the following definitions.

\begin{definition}
The \emph{reflected Gaussian density at a point} is
\[
\Theta_{refl(S)}(\mcfM, X) = \lim_{r \to 0} \Theta_{refl(S, \kappa)}(\mcfM, X, r), 
\]
where $\kappa$ is any number $\leq \min\{r_S/c_1, d(x, \partial U)\}$.
\end{definition}

\begin{definition}
Given $x \in S$, a \emph{reflected tangent flow} at $X = (x, t)$ is the flow $\tilde\mcfM'$ obtained by reflecting a tangent flow $\mcfM'$ at $X$ about its free-boundary planar barrier.
\end{definition}

Any reflected tangent flow of $\mcfM$ at $X$ will have constant Gaussian density equal to $\Theta_{refl(S)}(\mcfM, X)$.

We will need the following Proposition.  It should be standard.
\begin{prop}\label{prop:eucl-gauss-density}
Let $\mcfM \equiv (\mu(t))_{t \leq 0}$ be an ancient, self-shrinking Brakke flow in $\R^N$ with $\Theta(\mcfM) < \infty$, and write $\mu(-1) = \mu_{V}$.

Let $V'$ be a tangent cone to $V$ at $x$ (since $V$ is minimal in a weighted metric, we can take tangent cones at every point).  Then we have
\[
\Theta_{eucl}(V') = \Theta(\mcfM, (x, -1)) \leq \Theta(\mcfM).
\]
\end{prop}

\begin{proof}
Let $\lambda_i \to 0$ be chosen so that $(V - x)/\lambda_i \to V'$ as varifolds (where we interpret translation and dilation in the obvious sense of pushforwards).  By passing to a further subsequence, we can assume $\dilD_{1/\lambda_i}(\mcfM - (x, -1)) \to \mcfM'$.

Since $\mcfM$ is self-shrinking, $\mcfM'$ must be a static, eternal flow with $\mcfM'(t) \equiv \mu_{V'}$.  We calculate, using the coarea formula and standard formulas for $\omega_n$, 
\begin{align*}
\Theta(\mcfM, (x, -1)) = \Theta(\mcfM') &= \Theta(\mcfM', 0, 1) \\
&= (4\pi)^{-n/2} \int_{\mcfM'(-1)} e^{-\frac{|x|^2}{4}} \\
&= |V' \cap B_1| \frac{1}{n} (4\pi)^{-n/2} \int_0^\infty r^{n-1} e^{-r^2/4} dr \\
&= \Theta_{eucl}(V'). \qedhere
\end{align*}
\end{proof}

\section{Upper-semi-continuity}

The reduction of $\Theta_{refl}$ to the reflected Gaussian in the case of planar barriers, and the above monotonicity result, give good behavior in limits.

\begin{lemma}\label{lem:mono-blow-up}
Let $\mcfM_i$ be a sequence of Brakke flows in $\Omega_i \subset U_i$, so that (writing $S_i = \partial\Omega_i$) $\inf_i r_{S_i} \geq c_1\kappa > 0$.  Assume that
\[
\limsup_i \Theta_{refl(S_i)}(\mcfM_i, X_i, r_i) \leq \Theta_0
\]
for every $X_i \to 0$, $r_i \to 0$, subject to $r_i \leq R_i$ for some sequence $R_i$.

Let $\lambda_i \to \infty$ be another sequence, with $\lambda_i R_i \to \infty$, and suppose the dilated flows $\dilD_{\lambda_i}\mcfM_i$ converge to some Brakke flow $\mcfM'$ in $\R^N$, having possible free-boundary in a plane.

If $\tilde\mcfM'$ is the reflection of $\mcfM'$ across its barrier plane, then we have
\[
\Theta(\tilde\mcfM', X, r) \leq \Theta_0 \quad \forall X, r.
\]
\end{lemma}

\begin{proof}
Fix an $X_0 = (x_0, t_0)$, $r$, $R$.  By assumption we have a (possibly empty) affine plane $P$ so that $\lambda_i S_i \to P$ in $C^3_{loc}$.  Let $\refl$ be the affine reflection about $P$ if it exists, and formally $\infty$ if it does not.

We have by the dominated convergence theorem:
\begin{align*}
&\int_{\mcfM'(t_0 - r^2) \cap B_R} r^{-m} (e^{-|x - x_0|^2/4r^2} + e^{-|\refl(x) - x_0|^2/4r^2}) \\
&= \lim_i \int_{(\dilD_{\lambda_i} \mcfM_i)(t_0 - r^2) \cap B_R} f_{\lambda_i S_i, \lambda_i \kappa, X_0}(x, t_0 - r^2) \\
&\leq \limsup_i \Theta_{refl(\lambda_i S_i)}(\dilD_{\lambda_i} \mcfM_i, X_0, r) \\
&= \limsup_i \Theta_{refl(S_i)}(\mcfM_i, \dilD_{1/\lambda_i} X_0, r/\lambda_i) \\
&\leq \Theta_0
\end{align*}
The last inequality is justified because of the requirement $\lambda_i R_i \to \infty$.

Taking $R \to \infty$, we have by the monotone convergence theorem that
\[
\Theta(\tilde\mcfM', X_0, r) \leq \Theta_0  . \qedhere
\]
\end{proof}

\begin{lemma}\label{lem:mono-usc}
Let $\mcfM_i$ be a sequence of free-boundary Brakke flows in $\Omega_i \subset U_i$.  Suppose
\[
\inf r_{S_i} \geq c_1\kappa > 0, \quad U_i \to U, \quad \Omega_i \to \Omega \text{ in } C^3_{loc},
\]
and $\mcfM_i \to \mcfM$ for some free-boundary Brakke flow in $\Omega \subset U$.

Suppose $\mcfM$ is defined for times $> -1$.  Take $X = (x, t)$ with $t > -1$.  Then for every $X_i \to X$, $r_i \to 0$, we have
\[
\limsup_i \Theta_{refl(S_i, \kappa)}(\mcfM_i, X_i, r_i) \leq \Theta_{refl(S)}(\mcfM, X).
\]

In particular, we have
\[
\limsup_i \Theta_{refl(S_i)}(\mcfM_i, X_i) \leq \Theta_{refl(S)}(\mcfM, X).
\]
\end{lemma}

\begin{proof}
We can assume $X_i = X = 0$ by translating the $\mcfM_i$, $\mcfM$ in spacetime.  Then $\mcfM$ is defined for time $\geq T_0$ for some $T_0 < 0$.  We can therefore assume the $\mcfM_i$ are defined for time $\geq T_0/2$.

Let $\tau_1 = \min(-T_0/4, \tau_0)$.  Since $\mcfM_i \to \mcfM$, we can choose a uniform $M$ with
\[
M \geq ||\mcfM_i(-\tau_1)||(B_{\kappa/2}(0)) \quad \forall i, \quad M \geq ||\mcfM(-\tau_1)||(B_{\kappa/2}(0)).
\]

Then for every $r^2 \leq \tau_1$, we have by the dominated convergence theorem and Theorem \ref{thm:pretty-mono}
\begin{align*}
\limsup_i \Theta_{refl(S_i, \kappa)}(\mcfM_i, 0, r_i)
&= \limsup_i e^{A \sqrt{r_i}} \Theta_{refl(S_i, \kappa)}(\mcfM_i, 0, r_i) + AM r_i^2 \\
&\leq \limsup_i e^{ A \sqrt{r}} \Theta_{refl(S_i, \kappa)}(\mcfM_i, 0, r) + AM r^2 \\
&= e^{A r^{\sqrt{r}}} \Theta_{refl(S, \kappa)}(\mcfM, 0, r) + AM r^2 .
\end{align*}
Now take $r \to 0$.
\end{proof}

\begin{corollary}
In the notation of the above Lemma, we have that
\[
\spt \mcfM = \{ \Theta_{refl(S)}(\mcfM, \cdot) \geq 1 \} \text{ in } U \times (-1, \infty) \subset \R^{N,1} ,
\]
and $\spt \mcfM_i \to \spt \mcfM$ in the local Hausdorff sense in $U \times (-1, \infty)$.
\end{corollary}

\begin{proof}
To prove the first claim it will suffice to show $\Theta_{refl} \geq 1$ on a dense subset of $\spt \mcfM$.  For a.e. $t \geq -1$, $\mcfM(t) = \mu_{V(t)}$ for some integral varifold $V(t)$.  Then for $\mu_{V(t)}$-a.e. $x$, $\mu_{V(t)}$ has a tangent plane.  Pick any $(x, t)$ satisfying these two conditions.

Let $\mcfM'$ be a tangent flow of $\mcfM$ at $(x, t)$.  Then $\mcfM'(0) = k \haus^n \llcorner P$ for some $n$-plane $P$.  By Proposition \ref{prop:eucl-gauss-density}, we deduce
\[
1 \leq k \leq \Theta(\mcfM') = \Theta_{refl(S)}(\mcfM, (x, t)).
\]
This proves the first claim.  The second now follows the first, and upper semi-continuity of density.
\end{proof}

\section{Local regularity}

We prove an a priori regularity estimate for classical flows, as in White \cite{white:local-reg}.

\begin{theorem}\label{thm:brakke-reg}
For any $0 < \alpha < 1$, and $\kappa > 0$, there are $\epsilon(n, N, \kappa, \alpha)$, $C(n, N, \kappa, \alpha)$ so that the following holds:

Let $\mcfM$ be a smooth $n$-dimensional mean curvature flow in $U$ supported in $\Omega$, with (classical) free-boundary in $S = \partial\Omega$, with $r_{3,\alpha}(S) \geq \kappa$.  Suppose for some open $\spU \subset U \times \R \subset \R^{N,1}$, $\mcfM$ is proper in $\spU$, and
\begin{equation}\label{eqn:brakke-reg-hyp}
\Theta_{refl(S)}(\mcfM, X, r) < 1 + \epsilon, \quad \forall X \in \spU,\,\, 0 < r < d(X, \partial \spU), 
\end{equation}
then
\begin{equation}\label{eqn:brakke-reg-concl}
r_{2,\alpha}(\mcfM, X)^{-1} d(X, \partial \spU) \leq C .
\end{equation}
\end{theorem}

\begin{remark}
If we assume $r_{\ell+1,\alpha}(S) \geq \kappa$, then the same proof verbatim bounds the $r_{\ell,\alpha}$-regularity scale of $\mcfM$.  In this case $\epsilon, C$ may depend on $\ell$.
\end{remark}

\begin{proof}
We follow White.  Let $\bar\epsilon$ be the infimum over all $\epsilon > 0$ for which the Theorem fails.  We wish to show $\bar\epsilon > 0$.

We have a sequence of numbers $\epsilon_i \to \bar\epsilon$, so that for each $i$ there is a smooth mean curvature flow $\mcfM_i$, supported in some smooth domain $\Omega_i$, having classical free-boundary in $S_i = \partial\Omega_i \cap U_i$, and proper in some open $\spU_i \subset U_i \times \R \subset \R^{N,1}$.  This sequence satisfies:

\begin{enumerate}
\item[A)] $r_{3,\alpha}(S) \geq \kappa$, 

\item[B)] $\Theta_{refl(S_i)}(\mcfM_i, X, r) < 1 + \epsilon_i$ for every $X \in \spU_i$, and $0 < r < d(X, \partial \spU_i)$, 

\item[C)] $\sup_{X \in \spU_i} r_{2,\alpha}(\mcfM_i, X)^{-1} d(X, \partial \spU_i) \to \infty$.
\end{enumerate}

By shrinking the $\spU_i$ as necssary, we can assume that the quantity $C)$ is finite for each $i$, and that every $\spU_i$ is bounded.

Choose $X_i$ so that
\[
r_{2,\alpha}(\mcfM_i, X_i)^{-1} d(X_i, \partial \spU_i) \geq \frac{1}{2} \sup_{\spU_i} r_{2,\alpha}(\mcfM_i, \cdot)^{-1} d(\cdot, \partial\spU_i) ,
\]
and set $\lambda_i = r_{2,\alpha}(\mcfM_i, X_i)^{-1}$.

Define the rescaled sequence
\[
\mcfM'_i = \dilD_{\lambda_i}(\mcfM_i - X_i), \quad \Omega_i' = \lambda_i (\Omega_i - x_i), \quad \spU_i' = \dilD_{\lambda_i} (\spU_i - x_i).
\]
Then each $\mcfM'_i$ satisfies the hypothesis of the Theorem, with $\Omega_i'$, $\spU_i'$, and $\epsilon_i$, but additionally $r_{2,\alpha}(\mcfM'_i, 0) = 1$.

By our choice of $\lambda_i$ we have
\[
d(0, \partial \spU_i') = \lambda_i d(X_i, \partial \spU_i) \to \infty ,
\]
and
\begin{align*}
r_{2,\alpha}(\mcfM'_i, \cdot)^{-1} 
&\leq \frac{2d(0, \partial \spU'_i)}{d(X_i, \partial \spU'_i)} \\
&\leq \frac{2d(0, \partial \spU'_i)}{d(0, \partial \spU'_i) - |X_i|} \\
&\to 2 \quad \text{uniformly on compact sets} ,
\end{align*}
and $r_{3,\alpha}(S_i') \geq \lambda_i \kappa \to \infty$.

Therefore, after passing to a subsequence, the $\mcfM'_i$ converge locally in $C^{2}$ to some $C^2$ mean curvature flow $\mcfM'$, which is proper in $\R^{N,1}$.  The limit $\mcfM'$ either has no boundary, or is supported in a half-space $\Omega$, with free-boundary in a plane $P = \partial\Omega$.  In the latter case we have $\Omega_i \to \Omega$ in $C^{3,\alpha}_{loc}$.

Let $\tilde\mcfM'$ be the Brakke flow without boundary obtained by reflecting $\mcfM'$ about $P$ if it exists (Proposition \ref{prop:brakke-reflect}), or simply $\mcfM'$ if $P$ does not.  By the free-boundary condition and interior Schauder estimates, $\mcfM'$ is entirely smooth.

By Lemma \ref{lem:mono-blow-up}, we have
\[
\Theta(\tilde\mcfM', X, r) \leq 1 + \bar\epsilon, \quad \forall X, \,\, r.
\]

Suppose, towards a contradiction, that $\bar\epsilon = 0$.  Then by the standard monotonicity formula we must have (after a suitable rotation in space)
\[
\tilde\mcfM' = \{0\}^{N - n} \times \R^n \times (-\infty, T], 
\]
for some possibly infinite $T \geq 0$.

Therefore on any compact set and $i$ large we have
\[
r_2(\mcfM'_i, \cdot)^{-1} \to 0, \quad r_{2,\alpha}(\mcfM'_i, \cdot)^{-1} \leq 3,
\]
and so after a suitable rotation, 
\[
\mcfM'_i \cap (B^{N-n}_5 \times B^n_5 \times (-\infty, T_i]) \subset \graph(u_i)
\]
where $u_i$ is defined on $B^{n,1}_5 \cap (Q_i \times (-\infty, T_i])$, for some smooth domain $Q_i$ (a perturbation of $\Omega_i \cap (\R^n \times \{0\})$), and some $T_i \geq 0$ converging to $T$.

The $u_i$ satisfy
\begin{equation}\label{eqn:norms-of-u}
|u_i|_{2, B^{n,1}_5} \to 0, \quad [u_i]_{2,\alpha, B^{n,1}_5} \leq c(n, \alpha).
\end{equation}
We wish to show that $[u_i]_{2,\alpha, B^{n,1}_2} \to 0$, as this will contradict our normalization $r_{2,\alpha}(\mcfM'_i, 0) = 1$.

If $P$ lies outside $B^{N-n}_5 \times B^n_5$, then the proof reduces to the boundaryless case considered by White.  We shall therefore assume $S_i \cap (B^n_5 \times B^{N-n}_5) \neq\emptyset$ for all $i$ large.

Let $\Phi_i$ be the map \eqref{eqn:straighten} straightening out the barrier surface $S_i$, centered at any point in $S_i\cap (B^{N-n}_5 \times B^n_5)$.  By equations \eqref{eqn:reflect-bounds} we have that
\begin{equation}\label{eqn:norms-of-Phi}
|\Phi_i^{-1} - Id| \to 0, \quad |D\Phi_i^{-1} - Id| \to 0, \quad [\Phi_i^{-1}]_{2,\alpha} \to 0
\end{equation}
uniformly on compact sets.

So using Lemma \ref{lem:perturb-graph}, we have $\Phi^{-1}_i(\graph(u_i)) = \graph (\tilde u_i)$, where
\begin{equation}\label{eqn:perturbed-u}
\tilde u_i(x, t) = u_i(x + \xi(x, t), t) + \eta(x, t)
\end{equation}
is defined on some half-ball $B^{n,1}_4 \times (\text{half-space} \times (-\infty, T_i])$, with standard Neumann boundary conditions.  The $\xi$, $\eta$ satisfy estimates
\begin{equation}\label{eqn:norms-of-xi}
|\xi|_{2, \alpha, B^{n,1}_4} \to 0, \quad |\eta|_{2,\alpha, B^{n,1}_4} \to 0.
\end{equation}
It will therefore suffice to show that $[\tilde u_i]_{2,\alpha, B^{n,1}_3} \to 0$.

Each $\tilde u_i$ satisfies the graphical mean curvature flow equation in the pullback metric $\gamma_i = \Phi_i^*\delta$:
\[
\partial_t \tilde u_i - g^{pq} D_{pq} \tilde u_i = 0,
\]
where $g^{pq}$ is the inverse of the matrix $g_{pq} = \gamma_i(e_p + D_p \tilde u_i, e_q + D_q \tilde u_i)$.  From relations \eqref{eqn:norms-of-Phi}, \eqref{eqn:norms-of-u}, \eqref{eqn:norms-of-xi} and standard Holder relations \eqref{eqn:holder-product}, \eqref{eqn:sp-holder-comp}, we have
\[
[\delta^{pq} - g^{pq}]_{0, \alpha, B^{n,1}_{7/2}} \to 0.
\]

Therefore by the parabolic Schauder estimates with Neumann boundary conditions (e.g. Theorem 4.23 in \cite{lieberman:parabolic}), we deduce $[\tilde u_i]_{2,\alpha, B^{n,1}_3} \to 0$.  This gives the required contradiction.
\end{proof}

\begin{corollary}[White \cite{white:local-reg}]\label{cor:brakke-reg}
Let $\bar\epsilon > 0$ be as in the proof of Theorem \ref{thm:brakke-reg}.  Then:

\begin{enumerate}
\item[A)] If $\mcfM'$ is a proper, smooth mean curvature flow in $\R^N$, with $\Theta(\mcfM') < 1 + \bar\epsilon$, then $\mcfM'$ is flat.

\item[B)] There is a non-flat, proper, smooth mean curvature flow $\mcfM$ in $\R^N$ with $\Theta(\mcfM) = 1 + \bar\epsilon$.

\item[C)] The $\epsilon$ of Theorem \ref{thm:brakke-reg}. depends only on $n, N$, even for estimates on $r_{\ell,\alpha}$. (of course the constant $C$ may still depend on $\kappa, \alpha, \ell$).
\end{enumerate}
\end{corollary}

\begin{proof}
The proof is in White \cite{white:local-reg}.  Since it is very short we reproduce it here.

For A), one can check easily that for every $X = (x, t)$ and $r$, with $t \leq 0$, we have
\[
\Theta(\mcfM', X, r) \leq \Theta(\mcfM').
\]
Therefore if $\Theta(\mcfM') < 1 + \bar\epsilon$, we can apply Theorem \ref{thm:brakke-reg} to $\mcfM' \cap \{ t \leq 0 \}$ to deduce
\[
r_{2,\alpha}(\mcfM' \cap \{ t \leq 0 \}, X)^{-1} d(X, \partial B^{N,1}_R) \leq C
\]
for every $R$ and $X$.  Therefore $\mcfM''$ is flat, and so is $\mcfM'$.

The flow of B) is simply the smooth, proper limit flow obtained in the proof of Theorem \ref{thm:brakke-reg}  C) is immediate from the existence of $\mcfM$ of B).
\end{proof}

\begin{corollary}[Brakke \cite{brakke}, White \cite{white:local-reg}]\label{cor:least-density}
There is an $\eta(n, N) \leq \bar\eps$ so that if $\mcfM$ is a self-shrinking Brakke flow in $\R^N$, having $\Theta(\mcfM) < 1 + \eta$, then $\mcfM$ is flat.
\end{corollary}

\begin{proof}
Take $1 + \eta$ the lesser of $1 + \bar\eps$ (of Corollary \ref{cor:brakke-reg}), and the least density of any non-flat minimal $n$-cone in $\R^N$ (which is $ > 1$ by Allard \cite{allard:first-variation}).

Suppose $\mcfM$ is a self-shrinker in $\R^N$ having density $\Theta(\mcfM) < 1 + \eta$, and write $\mcfM(-1) = \mu_V$.  Recall that $V$ is minimal in a weighted metric.  So at each point the Euclidean density exists, and by Proposition \ref{prop:eucl-gauss-density} must satisfy $\Theta_{eucl} < 1 + \eta$.  By Allard's Theorem and our choice of $\eta$, $\mcfM(-1)$ must be completely smooth.

We can therefore apply Corollary \ref{cor:brakke-reg} part A) to the smooth, proper flow $\mcfM \cap \{ t \leq -1 \}$ to deduce $\mcfM$ is flat.
\end{proof}

\begin{remark}
The least density is attained either by the flow of Corollary \ref{cor:brakke-reg} part B), or by the static least density minimal cone.
\end{remark}

Theorem \ref{thm:brakke-reg} implies a Brakke regularity Theorem for smooth flows, and limits of smooth flow.  We require some definitions.
\begin{definition}
Let $\mathcal S(n, N)$ be the collection of $n$-Brakke flows $\mcfM$ in $\R^N$ with free-boundary satisfying:

\begin{enumerate}
\item[A)] $\mcfM$ is supported inside some smooth domain $\Omega \subset \R^N$, and has free-boundary in $\partial \Omega$, 

\item[B)] $r_{3,\alpha}(\partial\Omega) > 0$.
\end{enumerate}
\end{definition}

\begin{definition}
Given a free-boundary Brakke flow $\mcfM$, we let $r_{k,\alpha}(\mcfM, X)$ be the largest radius $r$ so that $\spt\mcfM \cap B^{n,1}_r(X)$ is a smooth, proper mean curvature flow with classical free-boundary (as defined in Section \ref{section:spacetime-flows}), with $r_{k,\alpha}(\spt\mcfM, X) \geq r$.  We say $X$ is a \emph{regular point} if $r_{2,\alpha}(\mcfM, X) > 0$.
\end{definition}

\begin{definition}
Let $\mcfM_i$, $\mcfM$ be a sequence of Brakke flows in $\mathcal S(n, N)$, supported in $\Omega_i$, $\Omega$, with free-boundary in $\partial\Omega_i$, $\partial\Omega$.  We say $\mcfM_i \to \mcfM$ \emph{as Brakke flows with free-boundary} if the following holds:

\begin{enumerate}
\item[A)] $\inf r_{3,\alpha}(\partial\Omega_i) > 0$, for some fixed $0 < \alpha < 1$, 

\item[B)] $\Omega_i \to \Omega$ in $C^{3,\alpha}_{loc}$, 

\item[C)] if $I_i$, $I$ are the time-domains of definition for $\mcfM_i$, $\mcfM$, then $I_i \to I$, and for each $t$ in the interior of $I$, $\mcfM_i(t) \to \mcfM(t)$ as Radon measures.
\end{enumerate}
\end{definition}

\begin{theorem}\label{thm:brakke-closure}
Let
\[
\mathcal D(n, N) = \left\{ \begin{array}{c} \text{$\mcfM \in \mathcal S(n, N)$, such that if any (reflected) tangent flow at $X$ is a} \\ \text{multiplicity-1 (quasi-)static plane, then $X$ is a regular point.} \end{array} \right\} .
\]
Then $\mathcal D$ is closed under convergence of free-boundary Brakke flows.

Further, there is an $\eta(n, N)$ so that if $\mcfM \in \mathcal D$ has free-boundary in $\partial\Omega$, and $\Theta_{refl(\partial\Omega)}(\mcfM, X) < 1 + \eta$, then $X$ is a regular point.
\end{theorem}

\begin{remark}
By definition, any smooth, proper mean curvature flow (with classical free-boundary) lies in $\mathcal D$.
\end{remark}

\begin{proof}
The second assertion is immediate from Corollary \ref{cor:least-density}: If $\Theta_{refl(\partial\Omega)}(\mcfM, X) < 1 + \eta$, then any (reflected) tangent flow $\mcfM'$ at $X$ must satisfy $\Theta(\mcfM') < 1 + \eta$ also (Theorem \ref{thm:tangent-flows}).  Choosing $\eta$ as in Corollary \ref{cor:least-density}, we must have that $\mcfM'$ is flat.

We prove the first assertion.  Let $\mcfM_i \in \mathcal D$ be a sequence converging to some $\mcfM \in \mathcal S$.  Suppose $\Theta_{refl(\partial \Omega)}(\mcfM, 0) = 1$.

We wish to show there is an open set $\spU$, and an $i_0$, so that
\begin{equation}\label{eqn:brakke-closure-density}
\Theta_{refl(\partial\Omega_i, \kappa)}(\mcfM_i, X, r) < 1 + \eta \quad \forall i > i_0, \quad \forall X \in \spU, \quad \forall r < d(X_i, \partial \spU) .
\end{equation}
Here $\kappa > 0$ is chosen so that $\kappa < \inf_i r_{3,\alpha}(\partial\Omega_i)$, where $\Omega_i$ is the domain supporting $\mcfM_i$. 

Suppose, towards a contradiction, \eqref{eqn:brakke-closure-density} fails.  Then (passing to a subsequence as necessary) we have a a sequence $R_i \to 0$, and $X_i \in \cap B_{R_i}^{N,1}(0)$, so that
\[
\Theta_{refl(\partial\Omega_i, \kappa)}(\mcfM_i, X_i, R_i) \geq 1 + \eta.
\]
But then since $X_i \to 0$, $R_i \to 0$, this contradicts upper semi-continuity (Lemma \ref{lem:mono-usc}).  Therefore \eqref{eqn:brakke-closure-density} must hold for some domain $\spU$.

We deduce that (for $i$ large) $\mcfM_i \cap \spU$ is regular, and satisfies the conditions of Theorem \ref{thm:brakke-reg}.  So we have an a priori estimate of the form
\[
\sup_{\spU} r_{2,\alpha}(\mcfM_i, \cdot)^{-1} d(\cdot, \partial\spU) \leq C ,
\]
independent of $i$.  By Lemma \ref{lem:boundary-reg-points} $0$ is a regular point of $\mcfM$.
\end{proof}

\begin{corollary}\label{cor:brakke-classical}
Let $F : M^n \times [0, T) \to \overline{\Omega} \subset \R^N$ be a smooth mean curvature flow, with classical free-boundary in $S = \partial\Omega$.  Let $\mcfM$ be free-boundary Brakke flow induced by $F$.

If for any $x$ we have $\Theta_{refl(S)}(\mcfM, (x, T)) < 1 + \eta$, then $(x, T)$ is a regular point of $\mcfM$.  In other words, $F$ extends smoothly up to time $T$ near $x$.
\end{corollary}

\begin{proof}
Choose a increasing sequence $T_i \to T$.  Each flow $\mcfM_i = \mcfM \cap \{ t \leq T_i \}$ is smooth, proper, and hence $\mcfM_i \in \mathcal D(n, N)$.  Further, it is clear that $\mcfM_i \to \mcfM$ as free-boundary Brakke flows.  By Theorem \ref{thm:brakke-closure}, $\mcfM \in \mathcal D(n, N)$ also.  Therefore, by assumption, $\spt\mcfM$ is smooth near $(x, T)$.
\end{proof}

\section{Elliptic regularization}

We adapt the elliptic regularization construction of Ilmanen \cite{ilmanen:elliptic-reg} to the free-boundary setting.  All the real work here is Ilmanen's or White's, we merely verify the constructions work with our notion of free-boundary Brakke flow.

Throughout this section we will assume $\Omega \subset \R^N$ is a domain with smooth boundary $S = \partial\Omega$, satisfying $r_{3,\alpha}(S) > 0$; and $\Sigma$ is an integral $n$-current in $\Omega$, with finite mass, compact support, and satisfying additionally $\haus^{n+1}(\spt\Sigma) = 0$.

The main result of this section is the following.
\begin{theorem}[an adaption of Ilmanen \cite{ilmanen:elliptic-reg}] \label{thm:existence}
Then there is a Brakke flow $\mcfM \equiv (\mu(t))_{t \geq 0}$ in $\R^N$ supported in $\Omega$ with free-boundary in $\partial\Omega$, so that
\begin{enumerate}
\item[A)] $\lim_{t \to 0} \mu(t) = \mu_\Sigma$ , 

\item[B)] $\mcfM \in \mathcal D(n, N)$, where $\mathcal D$ as in Theorem \ref{thm:brakke-closure},

\item[C)] there is an integral $(n+1)$-current $T$ in $\Omega \times [0, \infty)$ satisfying:
\begin{enumerate}
\item[i)] $\partial T = \Sigma$, 
\item[ii)] $||T\llcorner (\Omega \times B)|| \leq (|B| + |B|^{1/2}) ||\Sigma||$ for any interval $B$, 

\item[iii)] $\mu(t) \geq \mu_{\partial (T \llcorner \Omega \times (t, \infty))}$ for every $t \geq 0$.
\end{enumerate}
\end{enumerate}
(Ilmanen calls $\mcfM$ satisfying condition C) an ``enhanced motion'').

Further, if $\Sigma$ is smooth, embedded, with classical free-boundary in $\partial\Omega$, then for some $\delta(\Sigma) > 0$, $\mcfM \cap \{ 0 \leq t \leq \delta\}$ is smooth and proper.
\end{theorem}

From the uniqueness the classical free-boundary flow of hypersurfaces (see e.g. \cite{stahl:regularity}) we obtain directly
\begin{corollary}
If $\Sigma$ is smooth, and $N = n+1$, then the flow $\mcfM$ of Theorem \ref{thm:existence} coincides with the classical free-boundary mean curvature flow as long as it exists.
\end{corollary}

Define the functional
\[
I_\epsilon(P) = \frac{1}{\epsilon} \int e^{-z/\epsilon} d\mu_P
\]
on the space of integral $(n+1)$-currents in $\R^{N+1}$.  Here $z$ is the $\R$ component of $\R^N \times \R$.  Recall that $I_\eps$ is the area function for the metric
\[
g = e^{\frac{-2z}{(n+1)\eps}} \delta_{eucl} .
\]
In this metric, every plane $\{z = const\}$ is strictly convex, with mean curvature pointing in the $\partial_z$ direction.

\begin{definition}
Let $\admis$ be the space of integral $(n+1)$-currents $P$ in $\Omega \times \R$, for which $\partial P = \Sigma$ and $\spt P \subset \Omega \times [0, \infty)$.
\end{definition}

Clearly $\admis$ is closed under weak convergence.  If $P_i$ is a minimizing sequence for $I_\eps$, then since we have local mass bounds we can take a limit $P_i \to P_\eps$.  Then $P_\eps \in \admis$ also, and by lower-semi-continuity of mass $P_\eps$ minimizes $I_\eps$ in $\admis$.

Moreover, we have that
\begin{equation}\label{eqn:no-0-concentration}
\mu_{P_\eps}(\Omega \times \{0\}) = 0.
\end{equation}
This follows by White's varifold maximum principle \cite{white:varifold-maximum}, the strict convexity of the $\{z = 0\}$ plane, and our assumption that $\haus^{n+1}(\spt \Sigma) = 0$.

$P_\eps$ can be extended to an integer-multiplicity rectifiable current in $\R^{N+1}$ by restriction.  The extension will satisfy
\[
\mu_{P_\eps} = \mu_{P_\eps} \llcorner (\Omega\times \R),
\]
though in general the boundary will not a priori be anymore integral.  (Actually Gruter \cite{gruter:minz-problems} has shown that for minimizers such as $P_\eps$ the extension boundary is integral and locally finite, but we will not need this fact.)

Here's what will happen.  We define the translating solitons $P_\eps(t) = P_\eps - t/\eps$.  As $\eps \to 0$, these ``stretch out'' to become a $z$-invariant Brakke flow, with initial condition $\Sigma \times (0, \infty)$.  This gives the required Brakke flow.

On the other hand, we can normalize $P_\eps$ by scaling $z$ by $1/\eps$.  As $\eps \to 0$, the normalized currents $T_\eps$ essentially approach the spacetime track of the Brakke flow obtained previously.  In fact the limit will sit beneath the spacetime track.  For reasons intimately connected with non-uniquess, there may be a mass discrepency between the spacetime track and the limit.

\begin{prop}\label{prop:I_eps-variation}
Let $P_\eps$ minimize $I_\eps$ in $\admis$.  Then as an integral varifold $P_\eps$ has locally bounded variation in $\R^N \times (0, \infty)$, and satisfies
\begin{equation}\label{eqn:P_eps-stationary}
S^T(H_{P_\eps}) = -\partial_z^\perp/\eps \quad \mu_{P_\eps}-a.e. x \in \R^N \times (0, \infty).
\end{equation}
\end{prop}

\begin{proof}
The relation \eqref{eqn:P_eps-stationary} follows directly from the fact that
\[
\delta I_\eps(P_\eps)(X) = 0 \quad \forall X \in \cT(\partial\Omega, \R^N \times (0, \infty)) \cap C^1.
\]
The locally finite variation is a consequence of Proposition \ref{prop:fb-finite-var}.
\end{proof}

Write $\mathrm{\sigma_s}(x, z) := (x, z + s)$, and define the varifolds
\[
P_\eps(t) := (\sigma_{-t/\eps})_\sharp P_\eps.
\]
Then Proposition \ref{prop:I_eps-variation} shows the associated Radon measures
\[
t \in [0, \infty) \mapsto \mu_\eps(t) := \mu_{P_\eps(t)}
\]
form a Brakke flow supported in $\Omega \times (0, \infty)$ with free-boundary in $\partial\Omega \times (0, \infty)$.

Given an open interval $A \subset \R$, let us write $P_\eps(A) = P_\eps \llcorner (\R^N \times A)$.

\begin{prop}[Ilmanen section 4]\label{prop:decreasing-stuff}
Let $P_\eps$ minimize $I_\eps$ in $\mathcal C$.  Then
\begin{enumerate}
\item[A)] we have $I_\eps(P_\eps) \leq ||\Sigma||$;

\item[B)] for any interval $A \subset \R$, 
\[
\frac{1}{|A|} \int |\partial_z^T|^2 d\mu_{P_\eps(A)} \leq I_\eps(P_\eps) ;
\]

\item[C)] we have
\[
\frac{1}{\eps} \int |\partial_z^\perp|^2 d\mu_{P_\eps} \leq I_\eps(P_\eps) ;
\]

\item[D)] in particular, for any interval $A \subset \R$, we have $||P_\eps(A)|| \leq (|A| + \eps) ||\Sigma||$.
\end{enumerate}
\end{prop}

\begin{proof}
The proof is identical to that in Ilmanen, since all the relevant vector fields lie in $\cT(\partial\Omega \times \R, \R^N \times (0, \infty))$.  We provide an overview.  Part A) follows by plugging in $\Sigma \times [0, \infty)$ into $I_\eps$.

To prove B)-D), first one proves: for any $\delta > 0$, and $w - \delta \geq z \geq 0$, we have
\begin{equation}\label{eqn:delta-decr}
\delta^{-1} \int |\partial_z^T|^2 d\mu_{P_\eps(z, z + \delta)} - \delta^{-1} \int |\partial_z^T|^2 d\mu_{P_\eps(w, w + \delta)} \geq \eps^{-1} \int |\partial_z^\perp|^2 d\mu_{P_\eps(z + \eps, w)} .
\end{equation}
To achieve this plug the vector field $X = \eta(z) e^{z/\eps} \partial_z$ into $\delta I_\eps(P_\eps)$, for $\eta$ an appropriate piece-wise linear function supported on $(0, \infty)$.

Similarly, one can prove that: for any $\delta > 0$ we have
\begin{equation}\label{eqn:delta-at-0}
\delta^{-1} \int |\partial_z^T|^2 d\mu_{P_\eps(0, \delta)} \leq I_\eps(P_\eps).
\end{equation}
One uses a family of vector fields $X = \phi(z) \partial_z$, where $\phi$ looks like:
\[
\phi(z) = z/\delta \quad \text{ on } [0, \delta], \quad \phi(z) = 1 - (z-\delta)/L \quad \text{ on } [\delta, L].
\]

These vector fields are allowed because we know by Corollary \ref{cor:mass-bounds} that $P_\eps(0, a)$ is compactly supported for any $a < \infty$. 

By \eqref{eqn:no-0-concentration} we need only prove B), C) for intervals $A \subset (0, \infty)$.  B) follows directly from \eqref{eqn:delta-decr}, \eqref{eqn:delta-at-0}, and choosing an appropriate partition of $A$.  C) follows from \eqref{eqn:delta-decr}, \eqref{eqn:delta-at-0}, and the monotone convergence Theorem.  D) is then immediate.
\end{proof}

\begin{theorem}
There is a sequence $\mu_i \equiv \mu_{\eps_i}$ so that $\mu_i \to \mu$ as free-boundary Brakke flows.  Here $t \in [0, \infty) \mapsto \mu(t)$ is a free-boundary Brakke flow in $\Omega \times (0, \infty)$, with
\[
\mu(t)(\R^N \times A) \leq |A| ||\Sigma|| \quad \forall t \in [0, \infty), \text{ intervals }A \subset \R.
\]
\end{theorem}

From this point on we shall fix $\mu$ and $\mu_i$ as in this Theorem.

\begin{proof}
Proposition \ref{prop:decreasing-stuff} shows the masses $\mu_\eps(0)$ are uniformly locally bounded in $\eps$.  The existence of $\mu$ follows from the compactness Theorem \ref{thm:brakke-compactness}.  The bounds follow from lower-semi-continuity of mass, because
\[
\mu_\eps(t)(\R^N \times A) = ||P_\eps||(A + t/\eps) \leq (|A| + \eps) ||\Sigma|| . \qedhere
\]
\end{proof}

\begin{prop}[Ilmanen 8.5/8.8]\label{prop:z-invariance}
The flow $\mu$ satisfies:

\begin{enumerate}
\item[A)] For a.e. $t \geq 0$, we have that
\[
\mu(t) = \mu_{V(t)} = \mu_{W(t) \times \R},
\]
where $V(t)$, $W(t)$ are integral $(n+1)$- and $n$-varifolds respectivly.

\item[B)] If $\theta \in C_c((0, \infty), \R)$ satisfies $\int \theta = 1$, then
\[
W(t)(\psi(x, S)) = V(t)(\psi(x, S \oplus \langle \partial_z \rangle) \theta(z)), \quad \mu_{W(t)}(\phi(x)) = \mu_{V(t)} (\phi(x)\theta(z)) .
\]

\item[C)] The collection $t \in [0, \infty) \mapsto \bar\mu(t) \equiv \mu_{W(t)}$ defines a free-boundary Brakke flow in $\Omega$.

\item[D)] We have
\[
\bar\mu(t)(\R^N) \leq ||\Sigma|| \quad \forall t \geq 0.
\]
\end{enumerate}
\end{prop}

\begin{proof}
Follows directly as in Ilmanen \cite{ilmanen:elliptic-reg} 8.5 and 8.8, using Theorem \ref{thm:semi-decreasing} to prove $z$-invariance.  Conclusions A), B) implies the free-boundary condition on $\bar\mu(t)$.
\end{proof}

We set $\mcfM$ to be the flow $\bar\mu$.  We've already established $\mcfM$ is an integral, free-boundary Brakke flow in $\Omega$.  We wish to show $\mcfM(0) = \mu_\Sigma$.

Let $\kappa_\eps(x, z) := (x, \eps z)$, and define
\[
T_\eps := (\kappa_\eps)_\sharp P_\eps.
\]
Then by precisely the same computation of Ilmanen \cite{ilmanen:elliptic-reg} Section 8.10, for any open interval $A \subset \R$, 
\[
||T_\eps(A)|| \leq (|A| + \eps^2)^{1/2}(1 + |A|^{1/2}) I_\eps(P_\eps) .
\]
Here $T_\eps(A) := T_\eps \llcorner (\R^N \times A)$.

We can pass to a subsequence (WLOG $\eps_i$ also), so that $T_i \equiv T_{\eps_i} \to T$ as currents in $\Omega \times \R$.  Here $T$ is an integral $(n+1)$-current supported in $\Omega \times [0, \infty)$, with $\partial T = \Sigma$, and
\[
\mu_T(\R^N \times A) \leq |A|^{1/2}(1 + |A|^{1/2})||\Sigma|| \quad \forall \text{ intervals } A \subset \R.
\]
The family of currents $t \mapsto \partial T(t, \infty)$ is $1/2$-Holder continuous in the flat-norm.

In particular, one obtains
\begin{lemma}
If $\delta_i \to 0$, then $T_{\eps_i}(t + \delta_i, \infty) \to T(t, \infty)$.
\end{lemma}

The following Proposition and Corollary finishes the proof of Theorem \ref{thm:existence} parts A), C).

\begin{prop}
For every $t \geq 0$, we have $\bar\mu(t)  \geq \mu_{\partial T(t, \infty)}$.  Note we are taking the boundary of $T(t, \infty)$ as a current in $\Omega \times \R$.
\end{prop}

\begin{proof}
Using test functions $\phi \in C^2_c(\Omega, \R_+)$, the computation in Ilmanen 8.11 shows $\bar\mu(t) \llcorner \Omega \geq \mu_{\partial T(t, \infty)}$.  But since $T$ is taken as a current in $\Omega \times \R$ we have $\mu_{\partial T(t, \infty)}(\partial\Omega \times \R) = 0$.
\end{proof}

\begin{corollary}
We have $\lim_{t \to 0^+} \bar\mu(t) = \bar\mu(0) = \mu_{\Sigma}$.
\end{corollary}

\begin{proof}
The previous Proposition shows $\bar\mu(0) \geq \mu_\Sigma$.  Since $\bar\mu(0)(\R^N) \leq ||\Sigma||$, we must have equality.

By Theorem \ref{thm:semi-decreasing} we know
\begin{equation}\label{eqn:cont-up-to-0}
\mu_{\Sigma} = \bar\mu(0) \geq \lim_{t \to 0^+} \bar\mu(t).
\end{equation}
But since the mass $||\partial T(t, \infty)||$ is continuous in $t$, we have
\[
||\Sigma|| = \lim_{t \to 0} \mu_{\partial T(t, \infty)}(\R^N) \leq \lim_{t \to 0} \bar\mu(t)(\R^N) \leq ||\Sigma|| .
\]
We therefore we must have equality in \eqref{eqn:cont-up-to-0}.
\end{proof}

We prove Theorem part B).

\begin{lemma}
For any $x \in \overline{\Omega} \times (0, \infty)$, we have
\[
\Theta_{refl(\partial\Omega\times \R)}(\mu_\eps, (x, 0)) = \Theta_{eucl}(P_\eps, x).
\]

And consequently, $\mu_\eps \in \mathcal D(n+1, N+1)$.
\end{lemma}

\begin{proof}
One readily verifies that since the flow $\mu_\eps$ moves by translation, any tangent flow $\mcfM'$ at $(x, 0)$ is a static cone $V$, where $V$ is some tangent cone of $P_\eps$ at $x$.  Then by Theorem \ref{thm:tangent-flows} and Proposition \ref{prop:eucl-gauss-density} we have
\begin{align*}
\Theta_{eucl}(P_\eps, x) &= \Theta_{eucl}(V) \\
&= \Theta(\mcfM') \\
&= \left\{\begin{array}{l l} \Theta_{refl(\partial\Omega \times \R)}(\mu_\eps, (x, 0)) & x \in \Omega\times \R \\ 
 \frac{1}{2} \Theta_{refl(\partial\Omega \times \R)}(\mu_\eps, (x, 0)) & x \in \partial\Omega \times \R \end{array}\right.
\end{align*}

If some (reflected) tangent flow of $\mu_\eps$ at $(x, t)$ is a multiplicity-1 (quasi-)static plane, then by Allard \cite{allard:first-variation} (if $x \in \Omega \times  \R$) or Gruter-Jost \cite{gruter-jost:allard} (if $x \in \partial\Omega \times \R$), the above shows that $P_\eps$ is regular at $x + (0, t/\eps)$.  This proves $\mu_\eps \in \mathcal D(n+1, N+1)$.
\end{proof}

The $\mu_i(t)$ converge to $\mu(t)$ as free-boundary Brakke flows.  Therefore by Theorem \ref{thm:brakke-closure}, $\mu(t) \in \mathcal D(n+1, N+1)$ also.  By Proposition \ref{prop:z-invariance} it clearly follows that $\mcfM \equiv \bar\mu \in \mathcal D(n, N)$.

\subsection{Smooth initial data}

Let $\mcfM \in \mathcal D(n, N)$ be a free-boundary Brakke flow in $\Omega$ (write $S = \partial\Omega$), with smooth initial data $\mcfM(0) = \Sigma$.  We do not actually require $\mcfM$ to arise from elliptic regularization.  The key arguments in this section are due to White.

\begin{lemma}\label{lem:smooth-initial-data}
Take $x \in \Sigma$, and let $L = \{0\}^{N-n} \times T_x\Sigma \times \R \subset \R^{N,1}$.  For any $\eps > 0$, there is a $r = r(x, \eps)$ so that
\[
\mcfM \cap (B^{N,1}_{r}(x, 0) \cap \{t > 0\}) = \graph(u), 
\]
for some smooth $u: \Omega \subset L \to L^\perp$, which satisfies
\begin{equation}\label{eqn:graph-estimates}
|u|/\sqrt{t} + |Du| + |D^2 u| \sqrt{t} + |\partial_t u|\sqrt{t} \leq \eps . 
\end{equation}
\end{lemma}

\begin{proof}
Suppose there is a sequence of points $Y_i = (y_i, t_i) \in \mcfM$, with $y_i \to x$, $t_i \to 0$, satisfying for each $i$ one of the following:
\begin{enumerate}
\item[A)] $Y_i$ is not a regular point, 

\item[B)] the projection $\proj_L$ restricted to $\mcfM$ is not a local diffeomorphism at $Y_i$, 

\item[C)] $\mcfM = \graph(u)$ near $Y_i$, but $|Du| + |D^2 u| \sqrt{t_i} + |\partial_t u| \sqrt{t_i}  \geq \eps$.
\end{enumerate}
We can of course assume a single condition fails for all $i$.

Consider the dilated flows
\[
\mcfM_i = \dilD_{1/\sqrt{t_i}}(\mcfM - (y_i, 0)).
\]
Notice condition C) is parabolic-scale-invariant.

Pass to a subsequence and obtain convergence $\mcfM_i \to \mcfM'$.  Since $\mcfM_i(0) = \frac{1}{\sqrt{t_i}} (\Sigma - y_i)$ converges to $T_x\Sigma$ in $C^1_{loc}$, we must have that $\mcfM'$ is a static multiplicity-1 plane.  If $x \in S$ then $\mcfM'$ has free-boundary in a plane.

This implies that $\limsup_i \Theta_{refl(S)}(\mcfM_i, X) = 1$ for any $X \in \R^N \times (0, \infty)$.  Therefore by Theorem \ref{thm:brakke-closure}, $\mcfM_i$ is regular at any such $X$ for $i$ sufficiently large.

Theorem \ref{thm:brakke-reg} implies the $r_{2,\alpha}(\mcfM_i)$ is uniformly bounded on compact subsets of $r_{2,\alpha}$.  Passing to a further subsequence as necessary, and using Lemma \ref{lem:boundary-reg-points}, we obtain smooth convergence $\mcfM_i \to \mcfM'$ in $\spU$.

In particular, for $i$ sufficiently large, we must have that
\[
\mcfM_i \cap (B^N_2(0) \times [-1/4, 4]) = \graph(u_i)
\]
for $u_i : \Omega_i \subset L \to L^\perp$ satisfying
\[
|Du_i| + |D^2 u_i| + |\partial_t u_i|  \to 0
\]
as $i \to \infty$.

But by construction we have $\tilde Y_i = \dilD_{1/\sqrt{t_i}} Y_i = (0, 1)$.  So $\tilde Y_i \in B^N_2(0) \times [-1/4, 4]$ is eventually graphical, with estimates \eqref{eqn:graph-estimates}, contradicting our intial choice.

This shows that near $x$, $\mcfM \cap \{0 < t\}$ splits as a union of graphs over $L$, each with estimate \eqref{eqn:graph-estimates}.  But by repeating the same blow-up argument with the dilates $\dilD_{1/\sqrt{t_i}}(\mcfM - x)$, we deduce $\mcfM$ must be one-sheeted.
\end{proof}

This shows that for some $\delta(\Sigma)$, $\mcfM$ is regular on $\{0 < t \leq \delta\}$, is $C^{0,1}$ (in spacetime) and $C^1$ (in space) up to $t = 0$.  This argument by itself is not sufficient to prove $C^\infty$ up to $t = 0$, since it only requires $\Sigma$ to be $C^1$.

We prove using a barrier argument that $\mcfM$ is $C^{1,1}$ up to $t = 0$.  Parabolic Schauder estimates will then give us $C^\infty$ (or, in general, as much regularity as $\Sigma$).

Choose $\kappa$ smaller than $r_S/c_1$ and $1/30$-th the $C^{1,1}$ regularity scale of $\Sigma$.  Given any $x \in \Sigma$, and unit vector $v$ in the normal bundle $N_x\Sigma$, we can attach a small ball $B_{x, v}$ of radius $\kappa$ passing through $x$ and having outward normal vector $-v$.

We wish to show that, for a short time $t \in [0, t_0(n, \kappa)]$, $\Sigma$ says disjoint from the ball $B_{x, v}(t)$ obtained by shrinking $B_{x, v}$ by the factor $1 - c(n, \kappa)t$.  This will imply $u$ as in Lemma \ref{lem:smooth-initial-data} satisfies the improved estimate
\begin{equation}\label{eqn:improved-est}
|u|/t + |Du|/\sqrt{t} + |D^2 u| + |\partial_t u| \leq 1
\end{equation}
in a sufficiently small spacetime neighborhood of $x$.  Then \eqref{eqn:improved-est} implies $\mcfM$ extends as in $C^{1,1}$ to time $0$, and therefore completes the proof of Theorem \ref{thm:existence}.

If $x \not\in B_{r_S/10}(\partial\Omega)$, then by considering the evolution of just $\phi(x, t)$ from Theorem \ref{thm:cutoff-evolution} we obtain the desired disjointness directly.

If $x \in B_{r_S/10}(\partial\Omega)$, then we may have to work a little harder.  By the free-boundary condition
\[
\tilde B_{x, v} = \{ \tilde x : x \in B_{x, v}\}
\]
is either disjoint from $\Sigma$ (if $x \not\in\partial\Sigma$), or touches tangentially at $x$ also (if $x \in \partial\Sigma$).  Let $y$ be the center of $B_{x,v}$.  If $y \in \overline\Omega$ we simply apply Theorem \ref{thm:cutoff-evolution} condition A).  Otherwise, we must apply Theorem \ref{thm:cutoff-evolution} condition B) to a small ball (and its reflection) centered at $\zeta(y)$, to ensure $\Sigma$ stays disjoint from a neighborhood of $y$.  Then we are justified in applying Theorem \ref{thm:cutoff-evolution} condition C) to deduce the required disjointness.  This completes the proof of Theorem \ref{thm:existence}.

\section{Appendix}

We show how the errors (in space) from straightening the barrier tranfer to errors in spacetime.  This is in principle standard but the spacetime nature of the perturbation makes it a little more confusing.

Recall the interpolation inequality: if $i' + 2j' \leq \ell$, then
\[
r^{i'+2j'}|D^{i'} \partial_t^{j'} u|_{0, B^{n,1}_r} \leq r^{\ell+\alpha} [u]_{\ell,\alpha, B^{n,1}_r} + |u|_{0, B^{n,1}_r} .
\]
So to control the $C^{\ell,\alpha}$ spacetime norm of $u$ in $B^{n,1}_r$, it suffices to control the top Holder semi-norms and the $C^0$ norm.  Of course the same kind of interpolation inequality holds for the standard Holder spaces.

We make use of the following identities:
\begin{align}
[f\circ g]_{\alpha, B^{n,1}_r} &\leq [f]_{\alpha, g(B^{n,1}_r)} (|Dg|_{0, B_r} + r |\partial_t g|_{0, B_r})^\alpha \label{eqn:sp-holder-comp1} \\
[f\circ g]_{\alpha, B^{n,1}_r} &\leq (|Df|_{0, g(B_r)} + r |\partial_t f|_{0, g(B_r)}) [g]_{\alpha, B^{n,1}_r} . \label{eqn:sp-holder-comp} \\
[f_1 \cdots f_k]_{\alpha, U} &\leq \sum_{i=1}^k [f_i]_{\alpha, U} \prod_{j\neq i} |f_j|_{0, U} \label{eqn:holder-product}
\end{align}

The following is a straightforward but tedious application of the inverse function theorem.
\begin{lemma}\label{lem:perturb-graph}
Let $u : U \subset B^{n,1}_r \to \R^{N-n}$ be $C^1$ in both variables, with
\[
r^{-1} |u| + |Du| + r |\partial_t u| \leq 1 .
\]
Suppose $\phi : B^N_{3r} \to \R^N$ satisfies
\[
\phi = Id + e, \quad r^{-1} |e| + |De| \leq \epsilon \leq \epsilon_1(n) .
\]

Then if we extend $\phi(x, t) := (\phi(x), t)$ to act on $\R^{N,1}$ we have $\phi(\graph(u)) = \graph (\tilde u)$, where
\[
\tilde u(y, t) = u(y + \xi(y, t), t) + \eta(y, t),
\]
with the estimates
\[
r^{-1} |\xi| + |D\xi| + r |\partial_t \xi| \leq c(n) \epsilon, \quad r^{-1} |\eta| + |D\eta| + r |\partial_t \eta| \leq c(n) \epsilon.
\]

If further we have
\[
r^{\ell + \alpha-1} [e]_{\ell, \alpha, B^n_r} \leq \epsilon, \quad r^{\ell+\alpha-1} [u]_{\ell,\alpha, B^{n,1}_r} \leq 1, 
\]
Then
\[
r^{\ell+\alpha-1} [\xi]_{\ell,\alpha, B^{n,1}_{r/2}} \leq c(n, \ell, \alpha) \epsilon, \quad r^{\ell+\alpha-1} [\eta]_{\ell,\alpha, B^{n,1}_{r/2}} \leq c(n,\alpha, \ell)\epsilon.
\]
\end{lemma}

\begin{proof}
Write $A(x, t) = x + e_1(x, u(x, t))$, so that
\[
\phi(x, u(x, t)) = (A(x, t), u(x) + e_2(x, u(x, t)).
\]
By assumption we have
\begin{align}\label{eqn:A-Dx}
r^{-1} |A - Id| \leq \epsilon, \quad |DA - Id| \leq |De_1|(1 + |Du|) \leq c(n) \epsilon,
\end{align}
and
\begin{equation}\label{eqn:A-Dt}
r |\partial_t A| \leq r |D e_1| |\partial_t u| \leq c(n) \epsilon.
\end{equation}

Therefore, by the inverse function theorem an inverse $A_t^{-1} \equiv A(\cdot, t)^{-1}$ exists for each time slice, and we can set
\begin{align*}
\tilde u(y, t) 
&= u(A^{-1}_t(y), t) + e_2(A^{-1}_t(y), u(A^{-1}_t(y, t))) \\
&=: u(y + \xi(y, t), t) + \eta(y, t).
\end{align*}
where
\[
\xi(y, t) = y - A^{-1}_t(y), \quad \eta(y, t) = e_2(A^{-1}_t(y), u(A^{-1}_t(y), t)) .
\]

From \eqref{eqn:A-Dx}, \eqref{eqn:A-Dt} we immediately obtain
\[
r^{-1}|A^{-1} - Id| \leq c(n) \epsilon, \quad |DA^{-1} - Id| \leq c(n) \epsilon, \quad r |\partial_t A^{-1}| \leq r|DA^{-1}| |\partial_t A| \leq c(n)\eps.
\]
This proves the $C^1$ estimates on $\zeta$.  The required $C^1$ estimate on $\eta$ follows similarly, e.g.:
\begin{align*}
|\partial_t \eta| &\leq |D e_2| ( |\partial_t A^{-1}| + |\partial_t u| + |Du| |\partial_t A^{-1}|) \leq c(n) \epsilon/r .
\end{align*}

To prove the higher order estimates on $\zeta, \eta$, we proceed as follows.  First, by an easy induction one can show that $D^\ell \partial^m_t (A^{-1})$ is a linear combination of terms involving $(DA)^{-1}$, $D^a e_1|_{(id, u)\circ A^{-1}}$, and $D^b \partial_t^c (id, u)|_{A^{-1}}$, where $|a| \geq 1$ in each term.

Using relation \eqref{eqn:sp-holder-comp1}, and our assumed bounds for $u$, we have
\begin{equation}\label{eqn:holder-e}
[(D^a e_i)\circ (id, u) \circ A^{-1}]_{\alpha, B^{n,1}_{r/2}} \leq c(n, \alpha) [D^a e_i]_{\alpha, B^n_r}
\end{equation}
for any $a$.  Notice the LHS is the spacetime Holder semi-norm, while the RHS is the regular Holder semi-norm.  This gives a Holder bound of the form
\[
r^{\ell+2m+\alpha-1} [D^\ell\partial_t^m A^{-1}]_{\alpha, B^{n,1}_{r/2}} \leq c(n, \ell, m, \alpha) \eps,
\]
which is the required estimate for $\zeta$.

By similar reasoning we have that $D^\ell\partial_t^m \eta$ is a linear combination of terms involving $D^a e_2|_{(id, u) \circ A^{-1}}$ and $D^b\partial_t^c ((id, u) \circ A^{-1})$.  Now use \eqref{eqn:holder-e}, \eqref{eqn:holder-product} and our assumed regularity of $u$ to obtain the Holder estimate on $\eta$.
\end{proof}

\begin{lemma}\label{lem:boundary-reg-points}
Let $\mcfM_i, \mcfM$ be a sequence in $\mathcal S(n,N)$, where $\mcfM_i \to \mcfM$ as free-boundary Brakke flows.  Suppose $0$ is a regular point of each $\mcfM_i$, and 
\[
\inf_i r_{2,\alpha}(\mcfM_i, 0)  > 0.
\]
Then $0$ is a regular point of $\mcfM$, and the $\mcfM_i$ converge to $\mcfM$ in $C^{2,\alpha}$ near $0$.

If, additionally, the barriers converge in $C^\infty$, then convergence near $0$ is smooth.
\end{lemma}

\begin{proof}
If $0$ is uniformly bounded away from the barriers $S_i$, then this follows from Arzela-Ascoli and interior Schauder estimates.  To handle points at the boundary we will straighten the barrier.

Take $0 \in S$, and by replacing $\mcfM_i$ with $\mcfM_i - (\zeta_i(0), 0)$ we can assume $0 \in S_i$ also.  Pass to a subsequence, rotate by a fixed amount in space, and replace $\mcfM_i$ with $\mcfM_i - (0, t_i)$ (with $t_i \to 0$) as necessary, and we have
\[
\mcfM_i \cap (B_\rho^{N-n} \times B_\rho^{n,1}) = \graph(u^{(i)}),
\]
where $u^{(i)} : Q_i \times I \subset B^{n,1}_\rho \to \R$ is uniformly bounded in $C^{2,\alpha}$, and $I$ is either the interval $[-\rho^2, \rho^2]$ or $[-\rho^2, 0]$.

Let $\Phi_i$, $\Phi$ be the map \eqref{eqn:straighten} straightening the barriers $S_i$, $S$ (resp.), centered at $0$.  Using Lemma \ref{lem:perturb-graph}, we have
\[
\Phi_i^{-1} (\graph(u^{(i)})) = \graph(\tilde u^{(i)}),
\]
where $\tilde u^{(i)}$ is uniformly bounded in $C^{2,\alpha}$ also.  There is a fixed half-space $H \subset \R^n$, so that $\tilde u^{(i)} : H \times I \to \R$ has Neumann boundary conditions in $\partial H$.

Arzela-Ascoli implies $\tilde u^{(i)}$ subsequentially converge in $C^{2,\alpha'}$ to some $\tilde u : H \times I \to \R$.  By definition of free-boundary convergence, $\Phi_i \to \Phi$ in $C^{2,\alpha}$.  We deduce that
\[
\mcfM \cap (B^{N-n}_\rho \times B^{n,1}_\rho) = \Phi(\graph(\tilde u)) =: \graph(u).
\]
This shows $\mcfM$ is proper and $C^{2,\alpha'}$ near $0$.

$\tilde u$ is a graphical mean curvature flow in the pullback metric $\gamma = \Phi^*\delta$.  Therefore $\tilde u$ satisfies
\begin{equation}\label{eqn:graphical-mcf-pullback}
\partial_t \tilde u - g^{kl} D_{kl} \tilde u = 0,
\end{equation}
where $g^{kl}$ is the inverse to the matrix $\gamma(e_k + D_k u, e_l + D_l u)$.  Provided $\rho$ is sufficiently small, by \eqref{eqn:reflect-bounds} this is a parabolic equation, with coefficients as regular as $D\tilde u$.  The usual bootstrap argument then gives $C^\infty$.

Suppose the barriers converge smoothly $S_i \to S$.  Each $\tilde u^{(i)}$ satisfies the graphical mean curvature equation \eqref{eqn:graphical-mcf-pullback}, with the pullback metric $\gamma_i = \Phi_i^*\delta$:
\[
\partial_t \tilde u^{(i)} = F(D\Phi_i, D\tilde u^{(i)}, D^2 \tilde u^{(i)})
\]
for some analytic $F$.  So the difference $w^{(i)} = \tilde u^{(i)} - \tilde u$ satisfies a linear PDE
\[
\partial_t w^{(i)} = a_{kl}^{(i)} D_{kl} w^{(i)} + b_k^{(i)} D_k w^{(i)} + c_k^{(i)} D_k(\Phi_i - \Phi) =: L w^{(i)}.
\]
Convergence of $\Phi_i$ and $w^{(i)}$ implies $L$ is uniformly elliptic, with constant terms going to $0$ in $C^{\infty}$.  The usual Schauder estimates then imply $w^{(i)} \to 0$ in $C^\infty$ also.
\end{proof}

The following boundary monotonicity formula appears in Allard \cite{allard:boundary}.

\begin{prop}[Allard \cite{allard:boundary}, Lemma 3.1]\label{prop:boundary-mono}
Let $V$ be an integral $n$-varifold in $U$ with free-boundary in $S$.  For any $0 < \tau < \sigma \leq r_S$, and any $h \in C^1_c(U, \R)$, we have
\begin{align*}
&\sigma^{-1} \int_{B_\sigma(S)} h |D^T d|^2 + d D^T h \cdot D^T d + h d tr_V D^2 d + h d S^TH \cdot D d d\mu_V \\
&\quad  - \tau^{-1}\int_{B_\tau(S)} h |D^T d|^2 + d D^T h \cdot D^T d + h d tr_V D^2 d + h d S^TH \cdot D d d\mu_V \\
&= \int_{B_\sigma(S) \setminus B_\tau(S)} D^T h \cdot D^T d + h tr_V D^2 d + h S^TH\cdot D d d\mu_V.
\end{align*}
Here $d = d(\cdot, S)$.

In particular, by the dominated convergence theorem, 
\begin{align*}
&\sigma^{-1} \int_{B_\sigma(S)} h |D^T d|^2 d\mu_V - \lim_{\tau\to 0} \tau^{-1} \int_{B_\tau(S)} h|D^T d|^2 d\mu_V \\
&=  \int_{B_\sigma(S) \setminus S} (1-d/\sigma) (D^T h \cdot D^T d + h tr_V D^2 d + h S^TH \cdot D d) d\mu_V .
\end{align*}
\end{prop}

\begin{proof}
Let $X$ be the vector field
\[
X = \phi(d) h(x) d D d \equiv \phi(d) h(x) (x - \zeta(x)) ,
\]
where $\phi$ is a cutoff function to be determined.  We have
\[
div_V(X) = \phi' |D^T d|^2 h d + \phi d D^T h \cdot D^T d + \phi h |D ^T d|^2 + \phi h d tr_V D^2 d .
\]

Therefore, if
\[
I(\rho) = \int \phi(d/\rho) h(x) |D^T d|^2 d\mu_V,
\]
then
\[
I - \rho I' = -\int \phi d D^T h \cdot D^T d + \phi h d tr_V D^2 d + \phi h d S^TH \cdot D d\mu_V .
\]

Integrating the above relation between $\tau < \sigma$, and then taking $\phi \to 1_{[0, 1]}$, we obtain
\begin{align*}
&\sigma^{-1} \int_{B_\sigma(S)} h |D^T d|^2 d\mu_V - \tau^{-1} \int_{B_\tau(S)} h |D^T d|^2 d\mu_V \\
&= \int_\tau^\sigma \rho^{-2} \int_{B_\rho(S)} d D^T h \cdot D^T d + h d tr_V D^2 d +  h d S^TH \cdot D d d\mu_V d\rho .
\end{align*}

Apply the standard layer-cake formula to the measure
\begin{align*}
\nu(A) &= \int_A d D^T h \cdot D^T d + d h tr_V D^2d  +  h d S^T H \cdot D d d\mu_V ,
\end{align*}
to obtain
\[
RHS = \int_{B_\sigma(S)\setminus B_\tau(S)} d^{-1} d\nu -\sigma^{-1} \int_{B_\sigma(S)} d\nu + \tau^{-1} \int_{B_\tau(S)} d\nu ,
\]
which is the required equality.
\end{proof}

\bibliographystyle{plain}

\begin{thebibliography}{10}

\bibitem{allard:first-variation}
W.~Allard.
\newblock On the first variation of a varifold.
\newblock {\em Annals of Mathematics}, 95:417--491, 1972.

\bibitem{allard:boundary}
W.~Allard.
\newblock On the first variation of a varifold: boundary behavior.
\newblock {\em Annals of Mathematics}, 101:418--446, 1975.

\bibitem{brakke}
K.~Brakke.
\newblock {\em The motion of a surface by its mean curvature}.
\newblock Princeton University Press, 1978.

\bibitem{buckland}
J.~Buckland.
\newblock Mean curvature flow with free boundary on smooth hypersurfaces.
\newblock {\em J. reine angew. Math.}, 586:71--90, 2005.

\bibitem{me:convexity}
N.~Edelen.
\newblock Convexity estimates for mean curvature flow with free boundary.
\newblock {\em Advances in Mathematics}, 294:1--36, 2016.

\bibitem{giga-sato}
Y.~Giga and M.-H. Sato.
\newblock Neumann problem for singular degenerate parabolic equations.
\newblock 6:1217--1230, 1993.

\bibitem{gruter:minz-problems}
M.~Gruter.
\newblock Regularity results for minimizing currents with a free-boundary.
\newblock {\em J. reine angew. Math.}, pages 307--325, 1987.

\bibitem{gruter-jost:allard}
M.~Gruter and J.~Jost.
\newblock Allard type regularity results for varifolds with free boundaries.
\newblock {\em Annali della Scuola Normale Superiore di Pisa}, 13:129--169,
  1986.

\bibitem{huisken:monotonicity}
G.~Huisken.
\newblock Asymptotic behavior for singularities of the mean curvature flow.
\newblock {\em Journal of Differential Geometry}, 31:285--299, 1990.

\bibitem{ilmanen:elliptic-reg}
T.~Ilmanen.
\newblock {\em Elliptic Regularization and Partial Regularity for Motion by
  Mean Curvature}.
\newblock Memoirs of the American Mathematical Society, 1993.

\bibitem{kagaya}
T.~Kagaya.
\newblock Convergence of the allen-cahn equation with neumann boundary
  condition on non-convex domains.
\newblock 2017.
\newblock arXiv:1710.00526.

\bibitem{katso-koss-reitich}
M.~Katsoulakis, G.~Kossioris, and F.~Reitich.
\newblock Generalized motion by mean curvature with neumann conditions and the
  allen-cahn model for phase transitions.
\newblock {\em The Journal of Geometric Analysis}, 5:255--279.

\bibitem{koeller}
A.~Koeller.
\newblock On the singular set of mean curvature flows with neumann free
  boundary conditions.
\newblock 2010.
\newblock arXiv:1012.0601.

\bibitem{lambert:minkowski}
B.~Lambert.
\newblock The perpendicular neumann problem for mean curvature flow with a
  timelike cone boundary conditon.
\newblock {\em Trans. Amer. Math. Soc.}, 366:3373--3388, 2014.

\bibitem{lieberman:parabolic}
G.~Lieberman.
\newblock {\em Second Order Parabolic Differential Equations}.
\newblock World Scientific, 1996.

\bibitem{marquardt:thesis}
T.~Marquardt.
\newblock The inverse mean curvature flow for hypersurfaces with boundary.
\newblock thesis.

\bibitem{tonegawa:free-allen-cahn}
M.~Mizuno and Y.~Tonegawa.
\newblock Convergence of the allen-cahn equation with neumann boundary
  conditions.
\newblock {\em SIAM Journal on Mathematical Analysis}, 47:1906--1932, 2015.

\bibitem{sato}
M.-H. Sato.
\newblock Interface evolution with neumann boundary condition.
\newblock {\em Adv. Math. Sci. Appl.}, 4:249--264, 1994.

\bibitem{stahl:singularity}
A.~Stahl.
\newblock Convergence of solutions to the mean curvauture flow with a neumann
  boundary condition.
\newblock {\em Calc. Variations \& PDE}, 4:421--441, 1996.

\bibitem{stahl:regularity}
A.~Stahl.
\newblock Regularity estimates for solutions to the mean curvature flow with a
  neumann boundary condition.
\newblock {\em Calc. Variations \& PDE}, 4:385--407, 1996.

\bibitem{volkmann:thesis}
A.~Volkmann.
\newblock Free boundary problems governed by mean curvature.
\newblock thesis.

\bibitem{wheeler}
V.~Wheeler.
\newblock Non-parametric radially symmetric mean curvature flow with a free
  boundary.
\newblock {\em Math. Z.}, 276:281--298, 2014.

\bibitem{white:local-reg}
B.~White.
\newblock A local regularity theorem for mean curvature flow.
\newblock {\em Annals of Mathematics}, 161:1487--1519, 2005.

\bibitem{white:varifold-maximum}
B.~White.
\newblock The maximum principle for minimal varieties of arbitrary codimension.
\newblock {\em Communications in Analysis and Geometry}, 18, 2009.

\end{thebibliography}

\end{document}